\documentclass{article}

\usepackage{color}
\usepackage{epsfig}
\usepackage{amsmath,amsfonts,amssymb,amscd,amsthm,amsbsy,bbm, epsf,calc,  pdfsync}
\usepackage{latexsym}
\usepackage{enumerate}
\usepackage[subnum]{cases}
\usepackage{appendix}

\usepackage{cite}
\usepackage{hyperref}

\newtheorem{thm}{Theorem}[section]
\newtheorem{lem}[thm]{Lemma}
\newtheorem{cor}[thm]{Corollary}

\theoremstyle{definition}
\newtheorem{defin}[thm]{Definition}
\newtheorem{rmk}[thm]{Remark}

\DeclareMathOperator{\spt}{spt}
\DeclareMathOperator{\Id}{Id}

\renewcommand{\epsilon}{\varepsilon}

\newcommand{\R}{\mathbb{R}}

\DeclareMathOperator*{\affdim}{affdim}

\newcommand{\Leb}[1]{\lvert {#1}\rvert_{\mathcal{L}}}
\def\Omegabar{\bar{\Omega}}
\def\xbar{\bar{x}}
\def\Dbar{\bar{D}}
\def\pbar{\bar{p}}
\def\xhat{{x}}
\def\phat{{p}}
\def\xbarhat{{\xbar}}
\def\pbarhat{{\pbar}}

\def\mountaintilde{\tilde{m}}

\def\affinevalley{\mathcal{A}^{n-k}}
\def\affinecvalley{\mathcal{A}_{x_e}^{n-k}}

\def\curvedvalley{\mathcal{A}_{\ctil}^{n-k}}
\def\sublevelset{S}
\def\halfspace{H}
\def\util{\tilde{u}}
\def\ctil{\tilde{c}}

\def\base{\mathcal{W}}
\DeclareMathOperator{\aff}{aff}
\DeclareMathOperator{\cl}{cl}
\def\cyl{\mathcal{Q}}
\DeclareMathOperator{\interior}{int}
\def\cone{\mathcal{K}}
\def\ebar{\bar{e}}

\def\gbar{\bar {g}}
\def\conebar{\bar{\mathcal{\cone}}}

\def\cylbar{\bar{\cyl}}
\def\M{M}
\def\Mbar{\bar{M}}
\def\E{E}
\def\Ebar{\bar{E}}
\newcommand{\constOne}{\beta}

\newcommand{\constThree}{C_0}
\newcommand{\constFour}{C_1}
\newcommand{\coord}[2]{\left[#1\right]_{#2}}
\newcommand{\Omegacoord}[1]{\coord{\Omega}{#1}}
\newcommand{\Omegabarcoord}[1]{\coord{\Omegabar}{#1}}
\newcommand{\subdiffcoord}[1]{\coord{\partial_cu({#1})}{#1}}

\newcommand{\cExp}[2]{exp^c_{#1}({#2})}
\newcommand{\ctilExp}[2]{exp^{\ctil}_{#1}({#2})}
\newcommand{\MTWcoord}[4]{-(c_{#1#2, \bar{p}\bar{q}}-c_{#1#2, \bar{r}}c^{\bar{r}, s}c_{s, \bar{p}\bar{q}})c^{\bar{p}, #3}c^{\bar{q}, #4}}
\def\subzero{\sublevelset_0}
\def\subzerocoord{\coord{\subzero}{\xbar_0}}
\def\subzerohat{{\sublevelset}_0}
\def\subzerohatcoord{\coord{\subzerohat}{\xbarhat_0}}

\newcommand{\euclidean}[2]{{#1}\cdot{#2}}
\newcommand{\surface}[2][n-1]{\mathcal{H}^{#1}{\left(#2\right)}}

\DeclareMathOperator{\conv}{ conv}

\author{Young-Heon Kim\thanks{Department of Mathematics, University of British Columbia, Vancouver, Canada. Email: yhkim@math.ubc.ca} and Jun Kitagawa\thanks{Department of Mathematics, University of British Columbia, Vancouver, Canada,
and
Pacific Institute for Mathematical Sciences, Vancouver, Canada. Email: kitagawa@math.ubc.ca}}
\title{On the degeneracy of optimal transportation
\thanks{
Y.H.Kim's research is partially supported by 
Natural Sciences and Engineering
Research Council of Canada Grants 371642-09 as well as the  Sloan Research Fellowship 2012--2014.  \ Any opinions, findings and conclusions or recommendations
expressed in this material are those of authors and do not reflect the views
 of the
Natural Sciences and Engineering Research Council of Canada or of the  Alfred P. Sloan Foundation.\break
J. Kitagawa is partially supported by a Pacific Institute for the Mathematical Sciences Postdoctoral Fellowship.
\copyright 2012 by the authors.
}}
\begin{document}
\maketitle

\begin{abstract}
We extend a dimensional upper bound on how much an optimal transport map can degenerate for the quadratic transportation cost, originally due to Caffarelli, to cost functions that satisfy the curvature condition of Ma, Trudinger, and Wang. 
\end{abstract}

\section{Introduction}\label{section: introduction}

The present paper addresses how much an optimal transport map can degenerate in terms of the dimension of certain sets. In optimal transport theory, one considers two probability measures $\rho$ and $\bar \rho$, on domains $\Omega$ and $\bar \Omega$, respectively. There is a given transportation cost function $c: \Omega \times \bar \Omega \to \R$. One is interested in the optimal mapping $T$, a measurable map that minimizes the total transportation cost
\begin{align*}
\int_{\Omega} c(x, F(x)) d\rho (x)
\end{align*}
among all $F: \Omega \to \bar \Omega$ with $F_\#\rho = \bar \rho$.  For certain type of cost functions, including the distance squared cost on Riemannian manifolds, the existence and uniqueness of $T$,  is well established (see \cite{Bre91, Caf96, GM96, McC01, MTW05} ), and it is well known that the map $T$ is associated with a potential function, say $u$, called the \emph{Brenier solution}, and the map $T$ can be found from the (almost everywhere defined) gradient of $u$. 
In particular, 
if $c(x, \xbar) = -\euclidean{x}{\xbar}$ on $\Omega \times \bar \Omega \in \R^n \times \R^n$, then we have  $\nabla u (x)$ for almost every $x$, where $u$ is a convex function. This function $u$ solves the classical Monge-Amp\`ere equation (at least in an appropriate weak sense),
\begin{align*}
 \det(D^2 u(x)) = \frac{\rho(x)}{\bar \rho(\nabla u(x))}.
\end{align*}

Under appropriate geometric conditions, especially the Ma-Trudinger-Wang condition of the cost function and $c$-convexity of the domains, and when the source and target measures $\rho$ and $\bar \rho$ are bounded above and below, the optimal map $T$ is continuous, and $u$ is differentiable: see \cite{Urb97, Del91, Caf90, Caf92, Caf92a, MTW05, TW09, Loe09, Liu09, LT10, LTW10, FKM11}.  These geometric conditions are in fact necessary for regularity \cite{MTW05, Loe09}. Also, without both bounds on $\rho$ and $\bar \rho$, regularity can no longer be guaranteed: see \cite{Wan95}. 

 In this paper, our main interest is in the case when the map $T$ is not continuous, or equivalently when the potential function $u$ is not differentiable. We use the affine dimension of the subdifferential $\partial u$, to measure the degeneracy of the map $T$. Roughly speaking, this dimension measures how many directions the function $u$ will be non-differentiable, equivalently,  how much the optimal map $T$ can split a point and spread its image in the target domain, and thus how strong the discontinuity of $T$ can be. See Definitions~\ref{def:subdiff} and~\ref{def: c-subdiff} in Section~\ref{section: setup}.

%

For the Euclidean transportation cost $c(x, \xbar) = -\euclidean{x}{\xbar}$ 
 on $\R^n$,   it is known due to Caffarelli~\cite{Caf93} that if the source measure $\rho$ is bounded from above on $\Omega$ and the target measure $\bar\rho$ is bounded from below on its support $\spt \nu$, 
then for each point $x$ where $\partial u(x) \cap (\spt {\nu})^{\interior} \ne \emptyset$,  
the affine dimension of $\partial u (x)$ is \emph{strictly} less than $\frac{n}{2}$. A Pogorelov type counterexample, also presented in~\cite{Caf93}, shows that this bound is sharp. What Caffarelli actually showed is the same dimensional bound for the contact set, i.e. the set of points having the same image by the optimal map, which is equivalent to the bound on the subdifferential by reversing the role of the target and the source. 
%

We extend the dimension estimate of Caffarelli to the case where the cost function satisfies the Ma-Trudinger-Wang~\eqref{A3w}  condition \cite{MTW05, TW09}. 
The main theorem we prove in this paper is the following. The relevant definitions and conditions concerning $c$-convex geometry are given in Section~\ref{section: setup} below. The method presented here is a geometric approach, and it also provides a new proof for the case of the Euclidean quadratic cost function.
 \begin{thm}\label{thm: A3w subdifferential bound}
Consider two open sets $\Omega\subset \M$ and $\Omegabar\subset\Mbar$ in Riemannian manifolds $\M$ and $\Mbar$, and fix two probability measures $\mu$ on $\Omega$, and $\nu$ on $\Omegabar$ with bounded supports $\spt{\mu}$ and  $\spt{\nu}$. Assume $\mu$ is absolutely continuous with respect to Lebesque measure. 
Suppose that $c$ satisfies\eqref{A0}--\eqref{A2} and~\eqref{A3w} and that $\Omega$ and $\Omegabar$  are $c$-convex with respect to each other. Further assume that 
$\Omega$ is strongly $c$-convex with respect to $\spt{\nu}$. Finally, let $u$ be the Brenier solution of the optimal transportation problem with cost $c$ from $\mu$ to $\nu$, which  satisfies 
 \begin{equation}\label{eqn: upper bound}
\Leb{\partial_cu(E) \cap \spt \nu }\leq \Lambda \Leb{E} 
\end{equation}
for any measurable $E\subset \Omega$ and some constant $\Lambda >0$.

Then, for any $x\in \Omega$ such that $\partial_c u(x)\cap (\spt{\nu})^{\interior} \neq \emptyset$
\begin{equation*}
\affdim{\left(\partial u (x)\right)}<\frac{n}{2}.
\end{equation*}
where $\affdim$ is the affine dimension of a convex set.
\end{thm}
Note that the condition $\partial_c u(x)\cap (\spt{\nu})^{\interior} \neq \emptyset$ in the above theorem is necessary, since otherwise we have an easy example (see for example, \cite{CY09}) in the Euclidean case $c(x, \xbar) = -\euclidean{x}{\xbar}$, where   $\spt \mu = \{ x \in \R^2 \mid |x| \le 1\}$, $\spt \nu =\{ x \in \R^2 \mid 1\leq |x| \leq 2\}$, and the Brenier solution $u$ satisfies  $\partial u (0) = \{ x \mid |x| \leq1\}$, thus, $\affdim(\partial u (0))=2$. 

We also comment here on some variants of the condition~\eqref{A3w}. The first is~\eqref{A3s}, which is for example, satisfied by distance squared cost functions on the round sphere~\cite{Loe10} and its perturbation and quotients~\cite{Vil08, FR09, DG10, DG11, KM10, LV10, FRV10, FRV11, FRV12}. Another variation is~\eqref{NNCC} which is for example, satisfied by the distance squared cost on the Euclidean space and the products of round spheres~\cite{KM10} (see also~\cite{FKM10}). Both~\eqref{A3s} and~\eqref{NNCC} are strictly stronger versions of~\eqref{A3w}, and one does not imply the other.  See Remark~\ref{rmk: nncc/a3s} for the relevant definitions.

It is worth mentioning that under~\eqref{A3s} or~\eqref{NNCC}, the result in Theorem~\ref{thm: A3w subdifferential bound} is fairly easy to prove. If one assumes~\eqref{A3s}, the conditions in Theorem~\ref{thm: A3w subdifferential bound} above are enough to apply the methods in~\cite{Loe09} to obtain that the $c$-subdifferential is a singleton at each point in the domain. On the other hand, assuming~\eqref{NNCC}, by using the results of~\cite[Section 6]{FKM09} one can directly extend the method used in the proof of the main theorem in~\cite{Caf93} to obtain the desired estimate. 


However, under~\eqref{A3w} alone, it is not clear how to extend Caffarelli's proof. The main difficulty lies in the fact that Caffarelli's proof relies on combining the arithmetic--geometric mean inequality with the divergence theorem. In the case of a more general cost function $c$, there is a corresponding positive definite matrix that the arithmetic--geometric mean can be applied to, but  there seems to be no easy way to use the divergence theorem to continue the proof. 

The present paper is organized as follows: 
In Section~\ref{section: setup} we give various definitions and some standard conditions related to $c$-convex geometry and the optimal transportation problem. In Section~\ref{section: euclidean} we first provide a proof of the main result, but with the Euclidean cost function $c(x, \xbar)= -\euclidean{x}{\xbar}$ (Theorem~\ref{thm: euclidean thm}). This allows us to highlight the geometric framework of the proof without the technical difficulties posed by a nonlinear cost function. In Section~\ref{section: general c}, we provide the preliminary setup for the proof of Theorem~\ref{thm: A3w subdifferential bound}  for a general cost function $c$, and in the following Section~\ref{section: main proof} we give the actual proof. 
In Section~\ref{section: a3w example} in the Appendix,
we demonstrate an example of a cost function (which originally appeared in~\cite{TW09} and was communicated to the authors by Neil Trudinger)  that satisfies~\eqref{A3w}, but does not satisfy~\eqref{A3s} nor~\eqref{NNCC} .


\section{$c$-convex geometry}\label{section: setup} 
We suppose that $\Omega$ and $\Omegabar$ are two bounded, open domains contained in two Riemannian manifolds $\M$ and $\Mbar$ respectively (the bar does not represent the closure of a set, we will use the superscript $E^{\cl}$ to denote this instead).  
We will also use $x$ and $\xbar$ to denote points in $\Omega$ and $\Omegabar$ respectively, while $D$ and $\Dbar$ will either denote the differential of a function with respect to the $x$ or $\xbar$ variables respectively.
\begin{rmk}\label{rmk: coordinates notation}
Suppose we have fixed a coordinate basis on an $n$-dimensional vector space $V$. Then, for a point, say $p\in V$, we will use superscripts to denote the coordinate entries, and the notation $p'=(p^1,\ldots, p^k)$ and $p''=(p^{k+1},\ldots, p^n)$ to denote the first $k$ and last $n-k$ coordinate entries with respect to this basis respectively.
Moreover, by an abuse of notation, we will sometimes write $\xbar'$ and $\xbar''$ to denote the $n$-dimensional vectors $(\xbar', 0)$ and $(0, \xbar'')$. 

Also, we will write $\lvert v \rvert$ and  $\euclidean{v}{w}$ 
 to denote the standard Euclidean norm and inner product for vectors $v, w \in \R^n$. When the entries are vectors from arbitrary vector spaces (the two vectors may be contained in different spaces), it is understood that there are fixed coordinate systems (which will be clear from context) on the vector spaces, and we are identifying each vector with the point in $\R^n$ that gives its coordinate representation.

Finally, if $V$ is a vector space, the notation $\langle f,v\rangle$ will denote the action of $f\in V^*$ on the vector $v\in V$.
\end{rmk}

We now state a number of conditions, standard in the theory of optimal transportation (see~\cite{MTW05} and~\cite{TW09}, also \cite{Loe09, KM10, LV10,  FRV11} for more details on these conditions). Much of the notation that follows is similar in exposition to what is used in~\cite{Kit12} by the second author.

\noindent\underline{\textbf{Regularity of cost}}:
\begin{equation}\label{A0}\tag{A0}
c\in C^4(\Omega^{\cl}\times \Omegabar^{\cl}).
\end{equation}

\noindent\underline{\textbf{Twist}}:

$c$ satisfies condition~\eqref{A1} if the mappings
\begin{equation}
\begin{array}{rl}
\xbar & \mapsto -Dc(x_0, \xbar) \\
x & \mapsto -\Dbar c(x, \xbar_0)
\end{array} \label{A1}\tag{A1}
\end{equation}
are injective for each $x_0 \in\Omega$ and for each $\xbar_0 \in \Omegabar$.
Here, $D$ (resp. $\bar D$) denotes the differential in the $x$ (resp. $\xbar$) variable. 

We denote the inverses of the above mappings by $\cExp{x_0}{\cdot}$ and $\cExp{\xbar_0}{\cdot}$.
Notice that for the cost $c(x, \xbar) = -\euclidean{x}{\xbar}$ on $\R^n\times \R^n$, these mappings are both just the identity map.

\noindent\underline{\textbf{Nondegeneracy}}:

$c$ satisfies condition~\eqref{A2} if, for any $x\in\Omega$ and $\xbar\in\Omegabar$, the linear mapping given by
\begin{equation}
\begin{array}{rl}
-\Dbar D c(x, \xbar): T_{\xbar}\Omegabar \to T_{-Dc(x, \xbar)}\left(T^*_{x}\Omega\right)\cong T^*_x\Omega
\end{array}\label{A2}\tag{A2}
\end{equation}
is invertible
(hence, so is its adjoint mapping, $-D\Dbar c(x, \xbar): T^*_{\xbar}\Omegabar \to T_x\Omega$).

The next condition, originally introduced by Ma, Trudinger and Wang \cite{MTW05, TW09}, is crucial in the regularity theory of optimal maps and is actually a necessary condition for regularity as shown by Loeper in~\cite{Loe09}.

\noindent\underline{\textbf{MTW}}:

$c$ satisfies condition~\eqref{A3w} if, for all $x\in\Omega$, $\xbar\in\Omegabar$, and any $V\in T_{x}\Omega$ and $\eta\in T^*_{x}\Omega$ such that $\langle \eta,V\rangle=0$, we have
\begin{equation}\label{A3w}\tag{MTW}
\MTWcoord{i}{j}{k}{l}(x, \xbar)V^iV^j\eta_k\eta_l\geq 0.
\end{equation}
Here all derivatives are with respect to fixed coordinate systems on $\M$ and $\Mbar$, with regular indices denoting coordinate derivatives of $c$ with respect to the $x$ variable, and indices with a bar above denoting coordinate derivatives with respect to the $\xbar$ derivative. Also, a pair of raised indices denotes the matrix inverse.
\begin{rmk}\label{rmk: nncc/a3s}
There are two common, stronger versions of condition~\eqref{A3w}.

We say that $c$ satisfies condition~\eqref{A3s} if there exists some $\delta_0>0$ such that for all $x\in\Omega$, $\xbar\in\Omegabar$, and any $V\in T_{x}\Omega$ and $\eta\in T^*_{x}\Omega$ such that $\langle{\eta},{V}\rangle=0$, we have
\begin{equation}\label{A3s}\tag{$\textrm{MTW}_+$}
\MTWcoord{i}{j}{k}{l}(x, \xbar)V^iV^j\eta_k\eta_l\geq \delta_0\lvert V\rvert^2\lvert\eta\rvert^2.
\end{equation}

We say that $c$ satisfies condition~\eqref{NNCC} (see \cite{KM10})  if for all $x\in\Omega$, $\xbar\in\Omegabar$, and \emph{any} $V\in T_{x}\Omega$ and $\eta\in T^*_{x}\Omega$, we have
\begin{equation}\label{NNCC}\tag{NNCC}
\MTWcoord{i}{j}{k}{l}(x, \xbar)V^iV^j\eta_k\eta_l\geq 0
\end{equation}
(note that~\eqref{NNCC} removes the ``orthogonality condition'' of $\langle\eta,V\rangle=0$).

The condition \eqref{A3w} looks complicated but, as we will see below, it leads to some elegant geometric implications: see Lemma~\ref{lem: DASM} and Corolloraries~\ref{cor:c-convex c-subdifferential} and \ref{cor: local to global}.
\end{rmk}
We now give the definition of a $c$-convex function, along with its $c$-subdifferential.
\begin{defin}\label{def: c-convex functions}
 We say that a real valued function $u$ defined on $\Omega$ is \emph{$c$-convex} if for each point $x_0\in\Omega$, there exists at least one $\xbar_0\in\Omegabar$ and $\lambda_0\in\R$ such that
 \begin{align*}
 -c(x_0, \xbar_0)+\lambda_0&=u(x_0),\\
 -c(x, \xbar_0)+\lambda_0&\leq u(x)
\end{align*}
for all $x\in\Omega$. 
Any function of the form $-c(\cdot, \xbar_0)+\lambda_0$ satisfying the above relations is said to be a \emph{$c$-function that is supporting to $u$ from below at $x_0$}.
\end{defin}

\begin{defin}\label{def:subdiff}
For a semiconvex function $u$, its {\em subdifferental} at $x_0$ is defined by
\begin{align*}
 \partial u (x_0) = \{ p \in T^*_x \Omega  \mid   u(\exp_{x_0}{(v)} ) \ge u (x_0) + \langle v, p\rangle + o (|v|) \},
\end{align*}
where $\exp_{x_0}$ is the usual Riemannian exponential map at $x_0$.
 \end{defin}

\begin{defin}\label{def: c-subdiff}
If $u$ is a $c$-convex function and $x_0\in\Omega$, we define its \emph{$c$-subdifferential at $x_0$} by 
\begin{align*}
 \partial_c u(x_0):&=\{\xbar\in\Omegabar\mid -c(\cdot, \xbar)+\lambda\text{ is a }c\text{-function}\\
 &\qquad \text{supporting to }u\text{ from below at } x_0 \text{, for some } \lambda\in\R\}.
\end{align*}
Note there is the immediate inclusion, $\partial_c u (x) \subset \cExp{x} {\partial u (x)}$ for each $x \in \Omega$. 
If $E\subset \Omega$, we write
\begin{align*}
 \partial_cu(E):=\bigcup_{x\in E}{\partial_cu(x)}.
\end{align*}
\end{defin}
\begin{rmk}\label{rmk: sets in cotangent coordinates}
Given any $x\in\Omega$ and $\xbar\in\Omegabar$, we denote $\E\subset\Omega$ and $\Ebar\subset\Omegabar$ represented in the cotangent spaces at $x$ and $\xbar$ respectively, by
\begin{align*}
\coord{\E}{\xbar}:&=-\Dbar c(\E, \xbar),\\
\coord{\Ebar}{x}:&=-D c(x,\Ebar).
\end{align*}
\end{rmk}
\begin{defin}\label{def: c-convex sets}
If $E\subset\Omega$ and $\xbar\in\Omegabar$, we say \emph{$E$ is (strongly) $c$-convex with respect to $\xbar$} if $\coord{E}{\xbar}$ is a (strongly) convex subset of $T^*_{\xbar}\Omegabar$.

If $E\subset\Omega$ and $\Ebar\subset\Omegabar$, we say \emph{$E$ is (strongly) $c$-convex with respect to $\Ebar$} if $E$ is (strongly) $c$-convex with respect to every $\xbar\in\Ebar$.

Analogous definitions hold with the roles of $x$ and $\xbar$, and $E$ and $\Ebar$ reversed.
\end{defin}

Now, suppose we have fixed a point $\xbar_0\in\Omegabar$. Then, by~\eqref{A1} and~\eqref{A2}, the map $p\mapsto \cExp{\xbar_0}{p}$ is a diffeomorphism between $\Omegacoord{\xbar_0}$ and $\Omega$. We then define the modified cost function
\begin{align}
\ctil(p, \xbar):=c(\cExp{\xbar_0}{p}, \xbar)-c(\cExp{\xbar_0}{p}, \xbar_0)\label{eqn: modified cost}
\end{align}
for $p\in\Omegacoord{\xbar_0}$ and $\xbar\in\Omegabar$. Also, given a $c$-convex function $u$ on $\Omega$, we modify it by defining 
\begin{align}
 \util(p):=u(\cExp{\xbar_0}{p})+c(\cExp{\xbar_0}{p}, \xbar_0),\label{eqn: modified u}
\end{align}
it is easy to see that $\util$ is $\ctil$-convex, and moreover
\begin{align*}
 \xbar \in \partial_c u (\cExp{\xbar_0}{p}) \Longleftrightarrow \bar x \in \partial_{\ctil} \util (p).
\end{align*}
When dealing with these modified functions (and also~\eqref{eqn: mountain tilde}) defined on cotangent spaces, we will continue to use the notation $D$ and $\Dbar$ to denote the differential with respect to the $p$ and $\pbar$ variables respectively. It will be clear from the context whether the differentiation is with respect to variables on the original domains, or on the cotangent spaces. We note here that with this convention,
\begin{align}
 -D\ctil(p, \xbar)&=\left [-\Dbar Dc(\cExp{\xbar_0}{p}, \xbar_0)\right]^{-1}(-Dc(\cExp{\xbar_0}{p}, \xbar)+Dc(\cExp{\xbar_0}{p}, \xbar_0))\label{eqn: modified differential}\\
 &\in T_{\xbar_0}\Mbar.\notag
\end{align}
These modified functions will be useful when we carry out the proof of the main theorem.
 
A key result detailing certain geometric properties of $c$-convex functions was discovered by Loeper~\cite{Loe09} for domains in $\R^n$, further developed in~\cite{TW09a,KM10, LV10, FRV11}, and extended to domains in manifolds under suitable conditions.
\begin{lem}[Loeper's maximum principle~\cite{Loe09}]\label{lem: DASM}
Suppose $c$ satisfies~\eqref{A0}--\eqref{A2} and~\eqref{A3w}.
 Also assume $\Omega$ is $c$-convex with respect to $\Omegabar$.  
 Fix $x_0\in\Omega$ and let $\pbar: [0,1]\to \Omegabarcoord{x_0}$ be a line segment.  Then, for any $x\in\Omega$
\begin{align*}
&-c(x, \cExp{x_0}{\pbar(t)})+c(x_0, \cExp{x_0}{\pbar(t)})\\
&\leq \max \{-c(x, \cExp{x_0}{\pbar(0)}) +c(x_0,  \cExp{x_0}{\pbar(0)}), \\
& \ \ \ \  \ \ \ \ \ \ \ \  \ -c(x, \cExp{x_0}{\pbar(1)})+c(x_0, \cExp{x_0}{\pbar(1)})\}.
\end{align*}
An analogous inequality holds with the roles of $\Omega$ and $\Omegabar$ reversed.
\end{lem}

This lemma has several important consequences:

\begin{cor}\label{cor:c-convex c-subdifferential} \cite[Theorem 3.1]{Loe09}
Suppose that $c$ satisfies~\eqref{A0}--\eqref{A2} and~\eqref{A3w}, $u$ is a $c$-convex function, and $\Omega$, $\Omegabar$ are $c$-convex with respect to each other. Then,  for any $x_0\in\Omega$, 
the set $\partial_c u(x_0)$ 
is $c$-convex with respect to $x_0$.
\end{cor}

\begin{cor}\label{cor: local to global}\cite[Proposition 4.4]{Loe09}
Suppose that $c$ satisfies~\eqref{A0}--\eqref{A2} and~\eqref{A3w}, $u$ is a $c$-convex function, and $\Omega$ is $c$-convex with respect to $\Omegabar$.
 Then, for any $\xbar_0\in\Omegabar$ and $\lambda_0\in\R$, if the function $u(\cdot)+c(\cdot, \xbar_0)-\lambda_0$ has a local minimum at a point $x_0\in\Omega$, then 
\begin{align*}
 \xbar_0\in\partial_cu(x_0).
\end{align*}
\end{cor}

The following quite useful  consequence of Lemma~\ref{lem: DASM}  is first observed and utilized in \cite{FKM09, FKM11} and \cite{Liu09}.
\begin{lem}\label{lem: c-convex sublevelset}
Suppose that $c$ satisfies~\eqref{A0}--\eqref{A2} and~\eqref{A3w}, $u$ is a $c$-convex function, and $\Omega$, $\bar\Omega$ are $c$-convex with respect to each other.
Then, for any $\xbar_0\in\Omegabar$ and $\lambda_0\in\R$, the {\em $c$-sublevel set} 
\begin{align*}
 \{x\in\Omega\mid u(x)\leq -c(x, \xbar_0)+\lambda_0\}
\end{align*}
is $c$-convex with respect to $\xbar_0$. In terms of the modified function $\util$ (see  \eqref{eqn: modified u}), the function $p \mapsto \util(p)$ is sublevel set convex in $p$.

\end{lem}

\begin{defin}\label{def: brenier sol}
Suppose $c$ satisfies conditions~\eqref{A0}--\eqref{A2}, and $\mu$ and $\nu$ are probability measures defined on $\Omega$ and $\Omegabar$ respectively. A $c$-convex function $u$ is \emph{a Brenier solution for the optimal transportation problem with cost $c$ from $\mu$ to $\nu$} if 
\begin{align*}
 T_{\#}\mu=\nu,
\end{align*}
where the mapping $T(x)$ is given by
\begin{align*}
 T(x):=\cExp{x}{Du(x)} \text{ \ \ \ for a.e. $x\in \Omega$}.
\end{align*}
\end{defin}
\begin{rmk}\label{rmk: existence of brenier sol}
 It is well known (e.g. see \cite{MTW05, Vil09}) that if $\mu$ is absolutely continuous with respect to volume measure on $\M$, and $c$ satisfies~\eqref{A0}--\eqref{A2}, there exists a unique Brenier solution (defined a.e.) up to translation, of the optimal transportation problem with cost $c$ from $\mu$ to $\nu$.
\end{rmk}

\section{Euclidean case}\label{section: euclidean}
In this section, we give a proof for the case $c(x, \xbar)=-\euclidean{ x}{\xbar}$ on $\R^n\times \R^n$.  (One should note that this case is equivalent to $c(x, \xbar) = \frac 1 2 |x - \xbar|^2$.)
This detour serves two purposes. First, our method gives an alternative proof of Caffarelli's original result~\cite{Caf93}.
Second,  we can better focus on illustrating key geometric ideas in our method in the Euclidean case, 
since in this special case we can bypass many of the obstacles that a nonlinear cost creates. 

First note that for this cost function, $c$-convexity of a set coincides with the usual notion of convexity, and a $c$-convex function is just a convex function. The $c$-subdifferential
$\partial_c u (x_0)$ at each point $x_0 \in \Omega$, coincides with the ordinary subdifferential $\partial u (x_0)$ as long as $\Omegabar$ contains $\partial u(x_0)$. This is the case if, for example, $\Omegabar\subset \spt{\nu}$ and $u$ is Brenier solution from $\mu$ to $\nu$. 
If $E\subset \R^n$, we will write
\begin{align*}
 \partial u(E):=\bigcup_{x\in E}{\partial u(x)}.
\end{align*}
 Moreover, by taking the Legendre transform twice, $u$ can be extended to a convex function on all of $\R^n$ with 
 $$\partial u (\R^n) \subset \conv (\spt \mu),$$ where $\conv(E)$ denotes the convex hull of  a set $E$. We denote this extension also by $u$. 


In the following we give an alternative proof of Caffarelli's result~\cite{Caf93}: 
\begin{thm}[see \cite{Caf93}]\label{thm: euclidean thm}
Consider probability measures $\mu$ and $\nu$ on $\R^n$ with bounded supports $\spt{\mu}$ and  $\spt{\nu}$. Assume $\mu$ is absolutely continuous with respect to Lebesque measure. Let $u$ be a convex potential on $\R^n$, that is a Brenier solution for the optimal transportation problem with cost $c(x, \xbar)=-\euclidean{x}{\xbar}$ from $\mu$ to $\nu$, which also satisfies 
 \begin{equation}\label{eqn: euclidean upper bound}
\Leb{\partial u(E) \cap \spt{\nu}}\leq \Lambda \Leb{E}
\end{equation}
for any measurable $E\subset \R^n$ and some constant $\Lambda >0$.
Then, for any $x\in \R^n$ with $\partial u(x)\cap (\spt{\nu})^{\interior} \neq \emptyset$,
\begin{equation*}
\affdim{\left(\partial u(x)\right)}<\frac{n}{2},
\end{equation*}
where $\affdim{E}$ is the affine dimension of a convex set $E$.
\end{thm}
\begin{proof} \ 
\subsubsection*{Step 0}
Suppose that the theorem does not hold.
Let 
\begin{equation*}
k:=\affdim{\left(\partial u(x_0)\right)}=\max{\left\{\affdim{\left(\partial u(x)\right)}\mid x\in\R^n\text{ and }\partial u(x)\cap (\spt {\nu})^{\interior}\neq \emptyset\right\}}
\end{equation*}
 for some $x_0\in \R^n$ with $k\geq\frac{n}{2}$. The idea of the following proof is to first find a point, say $x_e$, and two interrelated geometric shapes $\cyl_d $ in $\R^n$ (the domain of $\partial u$) 
  and $\cylbar_d $  in $\R^n$ (the target of $\partial u$),  depending on a small parameter $d$. Then by using the inequality~\eqref{eqn: euclidean upper bound}, we draw a contradiction as the parameter $d$ goes to zero.   
 
 \subsubsection*{Step 1}
 First, let us find a suitable  point $x_e$.  Note that by~\eqref{eqn: euclidean upper bound} we must have $k\leq n-1$. By subtracting an affine function, we may assume that the origin $0$ is contained in both the relative interior of $\partial u(x_0)$ and $(\spt{\nu})^{\interior}$, and that $u(x_0)=0$ (in particular, this implies $u \ge 0$), and by a rotation we may assume that the affine hull of $\partial u(x_0)$ is spanned by $\{e_i\}_{i=1}^k$, the standard orthonormal Euclidean basis. Now define the $(n-k)$-dimensional affine set
\begin{equation}\label{eqn: valley euclidean case}
\affinevalley:=\bigcap_{\xbar\in \partial u(x_0)}\{x\in\R^n\mid \euclidean{ (x-x_0)}{\xbar} =0\}=\{x\in\R^n\mid x'=0\}
\end{equation}
(see Remark~\ref{rmk: coordinates notation} for the notation $x'$ and $x''$).
Also define the contact set,
\begin{equation*}
\sublevelset_0:=\{x\in\R^n\mid u(x)=0\},
\end{equation*}
which is convex since $u\geq 0$.
 Note that $\sublevelset_0$, which is obviously closed, is in fact compact. This is from the boundedness of $\spt \mu$, the fact that $\partial u(x)\cap (\spt{\nu})^{\interior} \neq \emptyset$, and \eqref{eqn: euclidean upper bound}. For example, this  can be shown by applying the boundary-not-to-interior lemma~\cite[Theorem 5.1 (b)]{FKM11}, to a sufficiently large ball (a strongly convex set) containing $\spt {\mu}$. Since $\sublevelset_0$ cannot intersect the boundary of this ball, by connectedness it must remain bounded.

Now, by boundedness of $\sublevelset_0$, 
we may find a point  $x_e$ in $\partial \sublevelset_0$ with the exterior sphere property.  By a translation of the domain we may  assume $x_e$ to be the origin, i.e. $x_e=0$. Since $\partial u(x_0)$ contains the $k$-dimensional ball $\{\xbar\in\R^n\mid \lvert \xbar'\rvert\leq r_0,\ \xbar''=0\}$ for some $0<r_0\leq 1$, we also find that 
\begin{equation*}
0=u(x)\geq \pm\euclidean{ (x-x_0)}{\frac{r_0 \xbar}{\lvert \xbar\rvert}}
\end{equation*}
 for any $x\in\sublevelset_0$ and $\xbar\in\partial u(x_0)$, from which it follows that
\begin{equation}\label{eqn: contact set inclusion euclidean case}
\sublevelset_0\subset\affinevalley.
\end{equation}
Moreover, from the fact that $u(0) =0$, and the definitions of $\affinevalley$ and the subdifferential,  it follows 
\begin{equation*}
\partial u(x_0)\subset \partial u(0),
\end{equation*}
and by the maximality of $k$, we find that $\partial u(0)$ also has affine dimension $k$. 

\subsubsection*{Step 2}
Next, let us construct the geometric shapes $\cyl_d $ and $\cylbar_d$. 
Recall that $0$ is a point in $\partial \sublevelset_0$ with the exterior sphere property. 
Hence, by a rotation in $x''$, we may assume for some $R_0>0$ that  $\sublevelset_0$ is contained in an $(n-k)$-dimensional ball with radius $R_0$, i.e.  
\begin{equation}\label{eqn: euclidean flat contact set}
\sublevelset_0\subset \{x\in\R^n\mid x'=0,\ \lvert x''+R_0e_n\rvert \leq R_0\}.
\end{equation}

We now define some auxiliary sets.  
For each $0 < \rho < \pi$ ($\rho$ will be chosen later to be close to $\pi/2)$, define the $n$-dimensional cone
 \begin{equation*}
 \cone_\rho:=\{x\in\R^n\mid \lvert x\rvert\cos{\rho}\leq x^n\}
 \end{equation*}
 (here, $\rho$ is the angle of the opening of the cone). Then, for $d>0$ and $0<\rho<\pi$, define
 \begin{equation*}
 \base_{d, \rho}:=\{x\in\R^n \mid    \lvert x''\rvert= d\} \cap\cone_{\rho}\cap B_1(0)^{\cl}.
\end{equation*}

We next determine the dependance of  $\rho$ on the parameter $d$, as we wish to obtain  
 $\base_{d,  \rho} \cap \sublevelset_0 = \emptyset, 
$
or equivalently, $u(x)>0$ for all $x\in \base_{d,  \rho}$. 
To this end, fix $x \in \base_{d,  \rho}$.  
Note that if $x'\neq 0$, we have $x\not\in \sublevelset_0$ by~\eqref{eqn: euclidean flat contact set}. On the other hand  
 if $x'=0$
 we can calculate
 \begin{align*}
 \lvert x''+R_0e_n\rvert^2&=\lvert x''\rvert ^2+R_0^2+2R_0x^n\\
 &\geq R_0^2+d^2+2R_0\lvert x\rvert\cos{\rho}\\
 &= R_0^2+d^2+2R_0  d\cos{\rho},
 \end{align*}
 thus if we choose $\cos{\rho}>-\frac{ d}{2R_0}$ we obtain $\lvert x''+R_0e_n\rvert>R_0$, hence again $x\not \in \sublevelset_0$ by~\eqref{eqn: euclidean flat contact set}. 
Since for $d>0$ sufficiently small,
\begin{equation*}
\cos{\left(\frac{\pi}{2}+\frac{ d}{4R_0}\right)}=-\sin{\left(\frac{ d}{4R_0}\right)}> -\frac{ d}{2R_0},
\end{equation*}
we will have $u(x)>0$ for all $x\in \base_{d,  \rho_d}$, i.e. 
$
 \base_{d,\rho_d} \cap \sublevelset_0 = \emptyset, 
$
 for the choice 
\begin{align}\label{eqn: rho d}
\rho_d:= \frac{\pi}{2}+ \frac{ d}{4R_0}.
\end{align} 

Now, we construct $\cyl_d $ and $\cylbar_d$. First, define
\begin{equation*}
f (d):=\inf_{\base_{d, \rho_d} }  u(x)  
\end{equation*}
for each $d>0$ sufficiently small. 
Notice that by the previous paragraph,  $u>0$ on $\base_{d,\rho_d}$ which is also compact by definition, hence we see that for any $d>0$ small,
$$
f(d) >0. 
$$
 Observe that since $\spt \nu$ is bounded and $\partial u (\R^n) \subset \conv (\spt \mu)$,
$u$ is Lipschitz and therefore $f(d) / d \lesssim 1$.   
Then, define the set
\begin{equation*}
\cyl_d:=\{x\in\R^n\mid \lvert x'\rvert\leq f(d),\ \lvert x''\rvert \leq  d\} \cap\cone_{\rho_d}. 
\end{equation*} 
We also define 
 \begin{equation*}
\conebar_d:=\{\xbar\in\R^n\mid \lvert \xbar''\rvert\cos\Big( \frac{d}{4R_0}\Big)\leq \xbar^n\},
\end{equation*}
  a cylinder whose base is an $(n-k)$-dimensional cone.
Notice that the angle $\frac{d}{4R_0}$ here is complimentary to the angle $\rho_d$.  
Using this we define the set 
\begin{align}\label{eqn: Euclidean Wbar d}
\bar \base_d := \conebar_d \cap \{ \xbar\in \R^n \mid \lvert\xbar'\rvert\le \frac{r_0}{2},\ \lvert\xbar''\rvert = \frac{r_0f (d)}{2d}\} 
\end{align}
and then $\cylbar_d$ as the cone over  $\bar\base_d$ with vertex $0$, that is,
\begin{equation*}
\cylbar_d:= \bigcup_{\xbar \in \bar\base_d} [0, \xbar] 
\end{equation*}
where $[0, \xbar]$ denotes the line segment (in $\R^n$) connecting $0$ to $\xbar$. 

We can see that this set has the volume bound
\begin{align}\label{eqn: vol cylbar}
   \left(\frac{f(d)}{d} \right)^{n-k} d^{n-k-1}\sim \Leb{\cylbar_d}
\end{align}
where the constant of proportionality depends on $r_0$. 
Indeed, first note that for any $0\leq \lambda \leq \left(\frac{r_0f (d)}{2d}\right)$, it is easy to see that
\begin{align*}
 \cylbar_d\cap \{\lvert \xbar''\rvert=\lambda\}=\left\{\frac{2\lambda d}{r_0f (d)}\xbar\mid \xbar\in\bar\base_d\right\},
\end{align*}
which has Hausdorff dimension $n-1$, with surface measure
\begin{align*}
 \surface[n-1]{\cylbar_d\cap \{\lvert\xbar''\rvert=\lambda\}}&=\left(\frac{2\lambda d}{r_0f (d)}\right)^{n-1}\surface[n-1]{\bar \base_d}\\
 &\sim \left(\frac{f (d)}{d}\right)^{1-n}\lambda^{n-1}d^{n-k-1}\left(\frac{f (d)}{d}\right)^{n-k-1}.
\end{align*} 
Then by the coarea formula, 
\begin{align*}
 \Leb{\cylbar_d}&=\int_0^{\frac{r_0f (d)}{2d}}{\surface[n-1]{\cylbar_d\cap \{\lvert\xbar''\rvert=\lambda\}}d\lambda}\\
 &\sim d^{n-k-1}\left(\frac{f (d)}{d}\right)^{-k}\int_0^{\frac{r_0f (d)}{2d}}{\lambda^{n-1}d\lambda}\\
 &\sim\left(\frac{f (d)}{d}\right)^{n-k}d^{n-k-1}
\end{align*}
as claimed.

\subsubsection*{Step 3}
We now claim that for each $d>0$ sufficiently small, we have 
\begin{equation}\label{eqn: show this inequality euclidean case}
\euclidean{ x}{\xbar}\leq u(x)
\end{equation}
for $x\in \partial \cyl_d$ and $\xbar\in \cylbar_d$. 
Since $u\ge0$ and the function $\bar y\mapsto \euclidean{ x}{\bar y}$ is convex, by the definition of $\cylbar_d$ we can see that it is sufficient to show~\eqref{eqn: show this inequality euclidean case} for all $\xbar \in \bar\base_d$
(in the case of a more general cost function, the convexity of $\bar y \mapsto \euclidean{x}{\bar y}$ will be replaced by Loeper's maximum principle, Lemma~\ref{lem: DASM}).

 Fix $\xbar\in\bar\base_d$. As $$\partial \cyl_d \subset  \base_{d, \rho_d} \cup \{ x \in \R^n \mid  |x'| = f(d)\} \cup  \partial \cone_{\rho_d},$$
there are three cases for $x\in \partial \cyl_d$:
\begin{enumerate}
\item $x\in \base_{d, \rho_d} \cap \partial \cyl_d$;
\item $ x \in \partial \cyl_d$ and $\lvert x'\rvert=f(d)$ ;
\item $x\in\partial\cone_{\rho_d} \cap \partial \cyl_d$.
\end{enumerate}

\noindent {\em Case 1:}
First, suppose $x\in  \base_{d, \rho_d}\cap \partial \cyl_d$. Then since $r_0\leq 1$ and by the definition of $f(d)$,
\begin{align*}
\euclidean{ x}{\xbar}&=\euclidean{ x'}{ \xbar'} +\euclidean{ x''}{\xbar''} \\
&\leq \lvert x'\rvert\lvert \xbar'\rvert+\lvert x''\rvert\lvert \xbar''\rvert\\
&\leq  f(d) \cdot \frac{r_0}{2} + d\cdot\frac{r_0f(d)}{2d}\\
&\leq r_0f(d)\\
&\leq u(x).
\end{align*}

\noindent {\em Case 2:}
Next, suppose that $x\in \partial\cyl_d$ with $\lvert x'\rvert =f(d)$. 
Notice that by assumption,  $\frac{r_0 x'}{\lvert x'\rvert}\in\partial u(0)$. Thus, we have
\begin{align*}
\euclidean{ x}{\xbar}&\leq \lvert x'\rvert\lvert \xbar'\rvert+\lvert x''\rvert\lvert \xbar''\rvert\\
&\leq r_0 f(d)\\
&=r_0\lvert x'\rvert\\
&=\euclidean{ x}{ \frac{r_0 x'}{\lvert x'\rvert}}\\
&\leq u(x) . 
\end{align*}

\noindent {\em Case 3:}
Finally,  suppose that $x \in \partial \cone_{\rho_d} \cap \partial \cyl_d$. 
 Since the angles $\rho_d$ and $\frac{d}{4 R_0}$ are complementary, by the definitions of $\cone_{\rho_d}$ and $\conebar_d$ we have 
$\euclidean{ x}{ \xbar''} = \euclidean{ x''}{ \xbar''} \le 0 $.
Therefore, 
\begin{align}
\euclidean{ x}{ \xbar} &=\euclidean{ x}{ \xbar'} +\euclidean{ x}{ \xbar''}\notag\\
&\leq \euclidean{ x}{ \xbar'}\notag\\
&\leq u(x)\label{eqn: outside of cone euclidean case}.
\end{align}
 In the last line, we used the fact that $\xbar' \in \partial u (0)$ for $|\xbar '| \le r_0$. 
With this, we have shown the inequality~\eqref{eqn: show this inequality euclidean case} for all $x\in \partial \cyl_d$ and $\xbar\in \bar \base_d$, thus also for all $\xbar \in  \cylbar_d$.

\subsubsection*{Step 4}
We now claim that $\cylbar_d\subset\partial u(\cyl_d)$, which follows from the convexity of $u$ by showing that  the function $y \mapsto u(y)-\euclidean{ y}{ \xbar}$ has a local minimum inside $\cyl_d$ for each $\xbar\in \cylbar_d$. 
If this does not hold,  then, it implies $\euclidean{x}{\xbar}  \le u(x)$ for all $x \in \cyl_d$. However, by the calculations leading to~\eqref{eqn: outside of cone euclidean case}  we see that 
 $\euclidean{ x}{\xbar}\leq u(x)$ for $x$ in some small neighborhood of $0$ outside the cone $\cone_{\rho_d}$ thus outside $\cyl_d$. Since $u(0)=0$, these together shows that  $0$ is a local minimum, verifying the claim.

\subsubsection*{Step 5}

Finally, the inclusion $\cylbar_d\subset\partial u(\cyl_d)$ from the previous paragraph will lead to a contradiction. First note that since $0\in(\spt{\nu})^{\interior}$, this implies if $d>0$ is sufficiently small we have $\cylbar_d\subset (\spt{\nu})^{int}$. Then from~\eqref{eqn: euclidean upper bound} and considering the volumes of $\cyl_d$ and $\cylbar_d$, this implies for all $d>0$ sufficiently small (also recall~\eqref{eqn: vol cylbar})
 \begin{align*}
 d^{n-k-1}\cdot\left(\frac{f(d)}{d}\right)^{n-k} & \Leb{\cylbar_d}\\
&\lesssim\Leb{\partial u(\cyl_d)\cap \spt{\nu}}\\
&\lesssim \Leb{\cyl_d}\\
&\lesssim f(d)^{k}d^{n-k},
\end{align*}
%
 and after rearranging we obtain
\begin{equation*}
\left(\frac{f(d)}{d}\right)^{n-2k} \lesssim  d^{k+1}.
\end{equation*}

However, since $n-2k\leq 0$ by assumption and $f(d)/d$ is bounded from above, we see that $(f(d)/d)^{n-2k}$ is bounded below away from zero,  
 which is a contradiction as $d\to 0$.
  \end{proof}

\section{General costs with MTW condition -- set up for the proof of Theorem~\ref{thm: A3w subdifferential bound}}\label{section: general c}

In this section,  we prove some preliminary results in anticipation of the analogue of \textbf{Step 1} (from the proof of Theorem~\ref{thm: euclidean thm}), for when we undertake the proof of Theorem~\ref{thm: A3w subdifferential bound}.
%

When the Euclidean cost function is replaced by a more general cost function $c$ satisfying~\eqref{A0}--\eqref{A2} and~\eqref{A3w}, we may adapt the geometric argument used in the previous section to prove a result similar to Theorem~\ref{thm: euclidean thm}. However, the nonlinearity introduced by the cost function can cause a number of obstructions, requiring additional technical details. 

First, recall that as a preliminary step in the proof of Theorem~\ref{thm: euclidean thm} above, we ``shifted'' from the initial point $x_0$ where the Theorem was assumed to be violated, to a point in the contact set $\{u=0\}$ which possessed the exterior sphere property. In the Euclidean cost case we almost immediately obtained that the subdifferential of $u$ at the original point was also contained in the subdifferential at the ``shifted point,'' however for a general cost function we must carefully use Lemma~\ref{lem: DASM} to obtain such a containment. This is undertaken in Lemma~\ref{lem: c-subdifferential containment} below.

\begin{lem}\label{lem: c-subdifferential containment}
Suppose that $\Omega$ and $\Omegabar$ are $c$-convex with respect to each other, $c$ satisfies~\eqref{A0}--\eqref{A2} and~\eqref{A3w}, and $u$ is a $c$-convex function on $\Omega$. 
Let $\xhat_0\in \Omega$, 
and $\xbar_0\in \partial_cu(\xhat_0)$, and  
define  the contact set
\begin{equation*}
\subzerohat:=\{x\in\Omega\mid u(x)=-c(x, \xbarhat_0) +  c(x_0, \xbarhat_0) + u(x_0)\}.
\end{equation*}
Also suppose  $-Dc(\xhat_0, \xbar_0)$ is contained in the relative interior of the convex set $\subdiffcoord{\xhat_0}$. 
  Then, for any point  $x_e$ in $\subzerohat$, we have
\begin{equation*}
\partial_cu(\xhat_0)\subset\partial_cu(x_e).
\end{equation*}
\end{lem}
\begin{proof}
%
%
Fix a point $x_e\in\subzerohat$ and some $\xbar\in\partial_cu(\xhat_0)$. We first fix notation for the representations of $\xhat_0$, $x_e$ and $\xbar_0$ in the cotangent spaces at $\xbarhat_0$ and $x_0$ by writing
\begin{align*}
\phat_0:&=-\Dbar c(\xhat_0, \xbarhat_0)\in \Omegacoord{\xbarhat_0},\\
p_e:&=-\Dbar c(x_e, \xbarhat_0)\in \Omegacoord{\xbarhat_0},\\
\pbar_0:&=-Dc(\xhat_0, \xbar_0)\in\Omegabarcoord{\xhat_0}.
\end{align*}

With this notation and using the definitions in \eqref{eqn: modified cost} and \eqref{eqn: modified u} for $\ctil$ and $\util$, we easily see that 
\begin{align*}
 \util &\ge \util(p_0),\\
  \subzerohatcoord &=\{p\in\Omegacoord{\xbar_0}\mid\util(p)=\util(p_0)\},
\end{align*}
and the $\ctil$-function $-\ctil(\cdot, \xbar)+\ctil(p_0, \xbar)+\util(p_0)$ is supporting  to $\util$ from below at $p_0$.
We also define the sublevel set
\begin{equation*}
\mathcal{\halfspace}^{\ctil}_{\xbar}:=\{\phat\in \Omegacoord{\xbarhat_0}\mid -\ctil(\phat, \xbar)+\ctil(p_0, \xbar)\leq 0\}
\end{equation*}
and the (affine) half-space
\begin{align*}
\mathcal{\halfspace}_{\xbar}:&=\{\phat\in T^*_{\xbarhat_0}\Omegabar\mid \langle \phat-\phat_0, -D\ctil(p_0, \xbar)\rangle \leq 0\}.
\end{align*}
Now, by Lemma~\ref{lem: c-convex sublevelset} (and since $\util\geq \util(p_0)$ everywhere by construction), the sets $\subzerohatcoord$ and $\mathcal{\halfspace}^{\ctil}_{\xbar}$ are convex subsets of $T^*_{\xbarhat_0}\Omegabar$. Thus by~\eqref{A1} and~\eqref{A2}, and by a differentiation of $\ctil$ at $p_0$ it can be seen that $\mathcal{\halfspace}_{\xbar}$ is a 
supporting halfspace to the set $\mathcal{\halfspace}^{\ctil}_{\xbar}$ at $\phat_0$ (i.e. $\mathcal{\halfspace}^{\ctil}_{\xbar}\subseteq \mathcal{\halfspace}_{\xbar}\subsetneq T^*_{\xbarhat_0}\Omegabar$ and $\phat_0\in\partial\mathcal{\halfspace}^{\ctil}_{\xbar}\cap \partial\mathcal{\halfspace}_{\xbar}$). Additionally, since $p_e\in \subzerohatcoord$ we see that
\begin{equation*}
\util(\phat_0)=\util(p_e)\geq -\ctil(p_e, \xbar)+\ctil(p_0, \xbar)+\util(p_0),
\end{equation*}
 and in particular
 \begin{align*}
 p_e\in\mathcal{\halfspace}^{\ctil}_{\xbar}.
\end{align*}
 Now, since $\pbarhat_0$ is in the relative interior of $\subdiffcoord{\xhat_0}$ by assumption, for some $\delta>0$ we can find a $\ctil$-segment $\xbar (t)$, defined for $-\delta \le t \le 1$ with respect to $x_0$ with $\xbar$ at one end and $\xbar_0$ in the interior of the curve, namely,
  there exists a line segment $t \in [-\delta, 1] \mapsto \pbar (t) \in \subdiffcoord{x_0}$ such that  
\begin{align*}
& \xbar (t) := \ctilExp{p_0}{\pbar (t)}\in\partial_cu(x_0),\qquad\forall\;t\in[-\delta, 1];\\
&  \xbar_0 = \xbar (0)   \quad  \text{and} \quad \xbar = \xbar (1).
\end{align*}
Then by the same argument as above, 
\begin{align*}
 \util(p_0)\geq -\ctil(p_e, \xbar(-\delta))+\ctil(p_0, \xbar(-\delta))+\util(p_0)
\end{align*}
hence
\begin{align*}
 p_e\in \halfspace^c_{\xbar(-\delta)}:=\{\phat\in \Omegacoord{\xbarhat_0}\mid -\ctil(\phat, \xbar(-\delta))+\ctil(p_0, \xbar(-\delta))\leq 0\}.
\end{align*}
However, again by Lemma~\ref{lem: c-convex sublevelset}, the affine halfspace
\begin{align*}
\mathcal{\halfspace}_{\xbar(-\delta)}:&=\{\phat\in T^*_{\xbarhat_0}\Omegabar\mid \langle \phat-\phat_0, -D\ctil(p_0, \xbar(-\delta))\rangle \leq 0\}\\
&=\{\phat\in T^*_{\xbarhat_0}\Omegabar\mid \langle \phat-\phat_0,\; [-\Dbar Dc(x_0, \xbar_0)]^{-1}(\pbar(-\delta)-\pbar(0))\rangle \leq 0\}\\
&=\{\phat\in T^*_{\xbarhat_0}\Omegabar\mid \langle \phat-\phat_0,\; -\delta[-\Dbar Dc(x_0, \xbar_0)]^{-1}(\pbar(1)-\pbar(0))\rangle \leq 0\}\\
&=\{\phat\in T^*_{\xbarhat_0}\Omegabar\mid \langle \phat-\phat_0, -D\ctil(p_0, \xbar)\rangle \geq 0\}.
\end{align*}
is a nondegenerate supporting halfspace for $\halfspace^c_{\xbar(-\delta)}$ at $p_0$, which implies that  in particular we must have
\begin{align*}
 p_e&\in \{\phat\in T^*_{\xbarhat_0}\Omegabar\mid \langle \phat-\phat_0, -D\ctil(p_0, \xbar)\rangle \geq 0\}.
\end{align*}
Since this is the opposite halfspace to $\mathcal{\halfspace}_{\xbar}$, which is the supporting halfspace to $\mathcal{\halfspace}^{\ctil}_{\xbar}$, we must actually have
\begin{align*}
- \ctil(p_e, \xbar)+\ctil(p_0, \xbar)= 0
\end{align*}
hence
\begin{align*}
 - \ctil(p_e, \xbar)+\ctil(p_0, \xbar)+\util(p_e)&=\util(p_e),\\
  - \ctil(p, \xbar)+\ctil(p_0, \xbar)+\util(p_e)&=- \ctil(p, \xbar)+\ctil(p_0, \xbar)+\util(p_0)\\
  &\leq\util(p),\qquad\forall p\in\Omegacoord{\xbar_0}
\end{align*}
which implies
\begin{equation*}
\xbar \in \partial_cu(x_e).
\end{equation*}
Since $\xbar\in\partial_c u (x_0)$ was arbitrary, this finishes the proof. 
\end{proof}

Now, we will find such a ``shifted point'' $\xhat_e$, as discussed in the preceding paragraph of Lemma~\ref{lem: c-subdifferential containment}. Recall that in the proof of the Euclidean case (see Theorem~\ref{thm: euclidean thm}), we relied heavily on the various relationships between the $(n-k)$-dimensional affine set~\eqref{eqn: valley euclidean case} containing the contact set, and the subdifferential $\partial u(x_e)$. However, in the general case the contact set between $u$ and a supporting $c$-function, and the $c$-subdifferential of $u$ at a point are contained in different domains. Thus we must define two systems of coordinates (actually coordinates on the corresponding cotangent spaces), in such a way that an argument similar to that of the proof of Theorem~\ref{thm: euclidean thm} can be applied. 
We step through this construction in Lemma~\ref{lem: choosing a good point} below. The statement is rather long since we include all the relevant assumptions.
 
\begin{lem}\label{lem: choosing a good point}
Suppose that $\Omega\subset\M$, $\Omegabar\subset\Mbar$, $c$, $\mu$, and $\nu$ satisfy the same conditions as in Theorem~\ref{thm: A3w subdifferential bound}.

Now suppose that $\xhat_0\in\Omega$ is a point such that  $\partial_c u (\xhat_0) \cap (\spt \nu)^{int} \ne \emptyset$ 
 with
 \begin{align*}
  k:=\affdim{\left(\subdiffcoord{\xhat_0}\right)}=\max{\{ \affdim{\left(\subdiffcoord{x}\right)} \mid \partial_c u(x) \cap (\spt \nu)^{int} \ne \emptyset\}},
\end{align*}
and assume $k\ge 1$. 
Choose a point  $\xbar_0 \in \Omegabar$ such that  $-\Dbar c(\xhat_0, \xbar_0) \in \Omegabarcoord{x_0}$ is contained in the relative interior of the convex set $\subdiffcoord{\xhat_0}$.  Finally, define the contact set
\begin{align*}
\sublevelset_0:&=\{x\in\Omega\mid u(x)=-c(x, \xbar_0)+c(x_0, \xbar_0)+u(x_0)\}.
\end{align*}
Then, there exists a point $x_e \in \sublevelset_0$,  a basis $\{\ebar_i\}_{i=1}^n$ on $T^*_{x_e}\Omega$, and 
an orthonormal basis $\{e_i\}_{i=1}^n$ on $T^*_{\xbar_0}\Omegabar$, with coordinate systems on $T^*_{x_e}\Omega$ and $T^*_{\xbar_0}\Omegabar$  centered at the points 
\begin{align}\label{eqn: pbar 0 and p e}
\text{$\pbar_0:=-D c(x_e, \xbar_0) \in T^*_{x_e}\Omega$ and $p_e:=-\Dbar c(x_e, \xbar_0) \in T^*_{\xbar_0}\Omegabar$} 
\end{align}
 with coordinate directions given by $\{\ebar_i\}_{i=1}^n$ and $\{e_i\}_{i=1}^n$, respectively, and positive numbers $R_0$ and $r_0$ with the following properties (Recall Remark~\ref{rmk: coordinates notation} on notation using coordinates):
\begin{align}
\subzerocoord&\subset \{p\in \Omegacoord{\xbar_0} \subset T^*_{\xbar_0}\Omegabar \mid p'=0,\ \lvert p''+R_0e_n\rvert \leq R_0\},\label{eqn: flat contact set}\\
\euclidean{p}{\pbar}&=\langle p-p_e, [-\Dbar Dc(x_e, \xbar_0)]^{-1}(\pbar-\pbar_0)\rangle,\notag\\
&\qquad\forall p\in T^*_{\xbar_0}\Omegabar,\ \pbar\in T^*_{x_e}\Omega,\ \label{eqn: dual basis}\\
B^k_{r_0}(0):&=\{\pbar\in T^*_{x_e}\Omega\mid \lvert\pbar'\rvert\leq r_0,\ \pbar''=0\}\notag\\
&\subset \subdiffcoord{x_e},\label{eqn: subdifferential contains ball}\\
\mathcal{\halfspace}^{\ctil}_{\pm r_0\ebar_i}&\subset \{p\in \Omegacoord{\xbar_0}\mid \pm p^i\leq 0\},\qquad \forall 1\leq i\leq k\label{eqn: sublevelsets contained in halfspaces}.
\end{align}
Here, $\mathcal{\halfspace}^{\ctil}_{\pbar}$ is defined by~\eqref{eqn: c-halfspaces} below.

\end{lem}
\begin{proof}
Since $\Omega$ is strongly $c$-convex with respect to $\spt \nu$, and $\partial_cu(x_0)\cap (\spt{\nu}^{\interior}\neq \emptyset$, by \eqref{eqn: upper bound} we may use the boundary-not-to-interior lemma in \cite[Theorem 5.1 (b)]{FKM11} to see that the closure of $\subzero$ is contained in the open set $\Omega$, i.e. $\subzero^{\cl} \subset\Omega$.  
This shows in fact, $\subzero = \subzero^{\cl}$.
Thus, $\subzerocoord$ is closed and also by Lemma~\ref{lem: c-convex sublevelset} it is convex.  Therefore, there exists a point $p_e\in\partial\subzerocoord$ with the exterior sphere condition and with $p_e\in \Omegacoord{\xbar_0}^{\interior}$, which will lead to \eqref{eqn: flat contact set} after choosing appropriate coordinates below. 
Define
\begin{equation*}
x_e:=\cExp{\xbar_0}{p_e}\in\partial \sublevelset_0.
\end{equation*}
Now, since $\partial_cu(\xhat_0)\cap (\spt \nu)^{\interior}\neq \emptyset$ by assumption and $\partial_cu(\xhat_0)\subset\partial_cu(x_e)$ by Lemma~\ref{lem: c-subdifferential containment} above, we find $\partial_cu(x_e)\cap (\spt\nu)^{\interior}\neq \emptyset$ as well.
The inclusion  $\partial_cu(\xhat_0)\subset\partial_cu(x_e)$, and 
the maximality of $k$ implies
\begin{equation*}
\affdim{\left(\subdiffcoord{x_e}\right)}=k.
\end{equation*}
Also, since the map $-Dc(x_e, \cExp{\xbar_0}{\cdot}): \subdiffcoord{\xhat_0}\to \subdiffcoord{x_e}$ is a diffeomorphism, the point
\begin{align*}
\pbar_0:=-Dc(x_e, \xbar_0) 
\end{align*}
is contained in the relative interior of the convex set $\subdiffcoord{x_e}$. 

We must now determine an $(n-k)$-dimensional affine space containing the contact set, as in~\eqref{eqn: valley euclidean case} (we must, however, point out that the $(n-k)$-dimensional set determined here is only for the purpose of defining coordinates, in the proof of the actual theorem we will be forced to choose a curved analogue of~\eqref{eqn: valley euclidean case}). To this end, define for each $\pbar\in\subdiffcoord{x_e}$ the point
\begin{align*}
\xbar_{\pbar}:&=\cExp{x_e}{\pbar}\in\partial_cu(x_e).
\end{align*}
 We use the definitions \eqref{eqn: modified cost} and \eqref{eqn: modified u} for $\ctil$ and $\util$ (modified with respect to the fixed point $\xbar_0$), and 
consider  the $\ctil$-functions associated to $\xbar_{\pbar}$,
\begin{align}\label{eqn: mountain tilde}
\mountaintilde_{\pbar}(p):&=-\ctil(p, \xbar_{\pbar})+\ctil(p_e, \xbar_{\pbar})+\util(p_e),
\end{align}
the $\ctil$-sublevel sets
\begin{align}\label{eqn: c-halfspaces}
\mathcal{\halfspace}^{\ctil}_{\pbar}:&=\{p\in \Omegacoord{\xbar_0}\mid \mountaintilde_{\pbar}(p)\leq \util(p_e)\},
\end{align}
and affine half-spaces
\begin{align*}
\mathcal{\halfspace}_{\pbar}:&=\{p\in T^*_{\xbar_0}\Omegabar\mid\langle  p-p_e, -D\mountaintilde(p_e)\rangle\leq 0\}.
\end{align*}
As in the proof of Lemma~\ref{lem: c-subdifferential containment} above, by Lemma~\ref{lem: c-convex sublevelset}, each set 
$\mathcal{\halfspace}^{\ctil}_{\pbar}$ is convex in  $T^*_{\xbar_0}\Omegabar$, and by~\eqref{A1} and~\eqref{A2}, $\mathcal{\halfspace}_{\pbar}$ is a supporting halfspace for the set $\mathcal{\halfspace}^{\ctil}_{\pbar}$ at $p_e$ for each $\pbar\in \Omegabarcoord{x_e}$. Since 
\begin{align*}
\subzerocoord=\{p\in\Omegacoord{\xbar_0}\mid \util(p)=\util(p_e)\}
\end{align*}
and each $\mountaintilde_{\pbar}$ is supporting to $\util$, 
we find that
\begin{align*}
\subzerocoord&\subseteq \bigcap _{\pbar \in \subdiffcoord{x_e}}{\mathcal{\halfspace}^{\ctil}_{\pbar}}\\
&\subseteq \bigcap _{\pbar \in \subdiffcoord{x_e}}{\mathcal{\halfspace}_{\pbar}}\\
&=:\affinecvalley.
\end{align*}
Since $\pbar_0$ is in the relative interior of the $k$-dimensional set $\subdiffcoord{x_e}$, it follows that $\affinecvalley$ is an $(n-k)$-dimensional affine set.

We now choose an orthonormal basis (in the Riemannian metric $\gbar_{\xbar_0}$ on $T^*_{\xbar_0}\Omegabar$) $\{e_i\}_{i=1}^n$ of $T^*_{\xbar_0}\Omegabar$ such that $\{e_i\}_{i=k+1}^n$ span the linear subspace $\left(\affinecvalley-p_e\right)$. 
Also, since $p_e$ satisfies the exterior sphere property in $\subzerocoord$, we may choose the basis and coordinate directions so that~\eqref{eqn: flat contact set} is satisfied for some $R_0>0$. 

Next, we must determine the appropriate sense of ``orthogonality'' between the sets $\affinecvalley$ defined above, and $\subdiffcoord{x_e}$. Choose a collection of vectors $\{\ebar_i\}_{i=1}^n$ in $T^*_{x_e}\Omega$ given by the dual relations
\begin{align*}
\delta_{ij}&=\langle e_i, [-\Dbar Dc(x_e, \xbar_0)]^{-1}\ebar_j\rangle
\end{align*}
for each $1\leq i,\ j\leq n$. Since $\{\ebar_i\}_{i=1}^n$ is actually a basis by~\eqref{A2}, we may also define a coordinate system (denoted by $\pbar$) centered at $\pbar_0$ with positive coordinate directions given by $\{\ebar_i\}_{i=1}^n$. Then, for any $p\in T^*_{\xbar_0}\Omegabar$ and $\pbar\in T^*_{x_e}\Omega$, we  calculate
\begin{align*}
 \langle p-p_e, [-\Dbar Dc(x_e, \xbar_0)]^{-1}(\pbar-\pbar_0)\rangle&=\sum_{i,j=1}^n\langle p^ie_i, [-\Dbar Dc(x_e, \xbar_0)]^{-1}\pbar^j\ebar_j\rangle\\
 &=\sum_{i,j=1}^np^i\pbar^j\langle e_i, [-\Dbar Dc(x_e, \xbar_0)]^{-1}\ebar_j\rangle\\
 &=\euclidean{p}{\pbar}
\end{align*}
and hence obtain~\eqref{eqn: dual basis}. We easily see from the definition of $\affinecvalley$ that $\{\ebar_i\}_{i=1}^k$ spans the linear subspace $\left(\aff{\left(\subdiffcoord{x_e}\right)}-\pbar_0\right)$. 
Since $\pbar_0$ is in the relative interior of the convex set $\subdiffcoord{x_e}$, this implies that we obtain~\eqref{eqn: subdifferential contains ball} for some $r_0>0$. Finally, using~\eqref{eqn: dual basis} we calculate for any fixed $1\leq i\leq k$ that
 \begin{align*}
 \mathcal{\halfspace}^{\ctil}_{\pm r_0\ebar_i}&\subset\mathcal{\halfspace}_{\pm r_0\ebar_i}\cap\Omegacoord{\xbar_0}
\\
&=\{p\in \Omegacoord{\xbar_0}\mid\pm r_0\sum_{j=1}^n{p^j\langle e_j, [-\Dbar Dc(x_e, \xbar_0)]^{-1}\ebar_i\rangle}\leq 0\}\\
&=\{p\in \Omegacoord{\xbar_0}\mid\pm p^i\leq 0\}
\end{align*}
to obtain the final property~\eqref{eqn: sublevelsets contained in halfspaces}.
\end{proof}

\section{Proof of Theorem~\ref{thm: A3w subdifferential bound}}\label{section: main proof}
 We are finally ready to give the proof of Theorem~\ref{thm: A3w subdifferential bound}.
The results in Section~\ref{section: general c} allow us to make geometric constructions similar to those in the proof of Theorem~\ref{thm: euclidean thm}. Still, the nonlinearity of the cost function produces a number of additional difficulties. 

For one, the set $\affinecvalley$ from the proof of Lemma~\ref{lem: choosing a good point} above is not sufficiently geometrically motivated enough to fill the role of $\affinevalley$ from the proof of Theorem~\ref{thm: euclidean thm}. Instead, we must use $\curvedvalley$ (defined in~\eqref{eqn: valley A3w case} below) which is no longer an $(n-k)$-dimensional affine set, but a locally smooth $(n-k)$-dimensional submanifold of $\Omega$. We must carefully consider its relation to the coordinate system we defined above in Lemma~\ref{lem: choosing a good point}. Also, with a general cost $c$, we cannot expect the vanishing of terms due to orthogonality that we enjoyed at various points in the proof of Theorem~\ref{thm: euclidean thm}. Instead, we must  carefully keep track of the size of such error terms and show that they can be ultimately controlled. A side effect of this last difficulty is that we must alter the definition of $f(d)$ to account for the rate of decay of $u(x)$ as $x$ approaches the contact set.

\begin{rmk}\label{rmk: remark on constants}
In the interest of readability, we make the following comment concerning constants in the following proof. We will absorb all constants that depend only on the cost function, the domains $\Omega$ and $\Omegabar$, and the constants $n$, $k$, $r_0$, $R_0$ which are fixed at the beginning of the proof, into the single constant $M$. There are three constants $\constOne$, $\constThree$, and $\constFour$ which are introduced and left ``to be determined'',
we keep careful track of these constants throughout the proof. These constants are necessary in order to control various error terms that arise, due to the nonlinearity of the cost function. Eventually we will determine that it is sufficient to take $\constOne$, $\constThree$, and $\constFour$ to be small, depending on $M$ but independent of the parameter $d$.
\end{rmk}

\subsection{The Proof}
The proof is outlined similarly as the one for Theorem~\ref{thm: euclidean thm}.
\begin{proof}[{Proof of Theorem~\ref{thm: A3w subdifferential bound}}]\ 

\subsubsection*{Step 0}
Suppose that the theorem does not hold. Since $\partial_c u (x) = \cExp{x} {\partial u (x)}$ for each $x \in \Omega$ by Corollary~\ref{cor: local to global}, we see that this implies
\begin{align*}
k:&=\affdim{\left(\subdiffcoord{\xhat_0}\right)}\\
&=\max{\left\{\affdim{\left(\subdiffcoord{x}\right)}\mid x\in\Omega\text{ and }\partial_cu(x)\cap  (\spt \nu)^{int} \neq \emptyset\right\}} 
\end{align*}
 for some $\xhat_0\in \Omega$ with $k\geq\frac{n}{2}$. 
  Note that by~\eqref{eqn: upper bound} we must have $k\leq n-1$. 
 
\subsubsection*{Step 1}
 We now apply Lemma~\ref{lem: choosing a good point} above to obtain points $x_e\in \Omega$, $\xbar_0\in \Omegabar$  (and corresponding $p_e$ and $\pbar_0$, see \eqref{eqn: pbar 0 and p e}), 
 coordinate bases on $T^*_{\xbar_0}\Omegabar$ and $T^*_{x_e}\Omega$, denoted by $\{e_i\}_{i=1}^n$ and $\{\ebar_i\}_{i=1}^n$ respectively, and positive numbers $R_0$ and $r_0$. 
 Instead of working with the original $c$ and $u$, we will use in the following, the modified functions $\ctil$ and $\util$ with respect to the point $\xbar_0$ (see the definitions \eqref{eqn: modified cost} and \eqref{eqn: modified u} ).
  By adding a constant  we may assume $\util(p_e)=0$,  and by translations of the domains assume  that both
  \begin{align*}
  \text{$p_e$ and $\pbar_0=0$}
  \end{align*}
   (hence both coordinate systems on $\Omegacoord{\xbar_0}$ and $\Omegabarcoord{x_e}$ are centered at the origin).
In particular, this implies
$$
\util(0) = 0 ,  \  \util \ge 0,  \ \text{ and } \ 0 \in  \partial_{\ctil} \util (0).
$$
 
   Also, by an abuse of notation we will identify a point $p\in \Omegacoord{\xbar_0}$ or $\pbar\in \Omegabarcoord{x_e}$ with the coordinates given by $\{e_i\}$ or $\{\ebar_i\}$. In particular, note that by property~\eqref{eqn: dual basis} we have
\begin{equation}\label{eqn:p dot pbar}
 \euclidean{p}{ \pbar}=\langle p, [-\Dbar Dc(x_e, \xbar_0)]^{-1}\pbar\rangle
\end{equation}
for any $p\in \Omegacoord{\xbar_0}$ and $\pbar\in \Omegabarcoord{x_e}$, while by~\eqref{A1} and~\eqref{A2}, the length of any vector $\pbar$ in the Riemannian metric on $T^*_{x_e}\Omega$ is comparable to the Euclidean length $\lvert \pbar\rvert$ measured in these coordinates. In the remainder of the proof, it is helpful to keep in mind that $\pbar' \in \partial_{\ctil} \util(0)$ for any $|\pbar'| \le r_0$.

\subsubsection*{Step 2}

 Again as in the proof of Theorem~\ref{thm: euclidean thm}, we aim to construct two families of sets $\cyl_d\subset \Omegacoord{\xbar_0}$ and $\cylbar_d\subset\Omegabarcoord{x_e}$ depending on a small parameter $d$: the ultimate aim is to apply inequality~\eqref{eqn: upper bound} and obtain a contradiction in the volume comparison as $d$ approaches zero. First, define the set (with notation as in Lemma~\ref{lem: choosing a good point} and \eqref{eqn: mountain tilde} above)
\begin{equation}\label{eqn: valley A3w case}
\curvedvalley:=\bigcap_{i=1}^k\{p\in\Omegacoord{\xbar_0}\mid \mountaintilde_{r_0\ebar_i}(p)=0\}
\end{equation}
(compare with~\eqref{eqn: valley euclidean case}): by transversality of each level set $\{p\in\Omegacoord{\xbar_0}\mid \mountaintilde_{r_0\ebar_i}(p)=0\}$, $\curvedvalley$ is a locally $n-k$-dimensional submanifold, but not necessarily flat.

 Choose 
 \begin{align*}
 \rho_d  := \frac{\pi}{2}+ \frac{ d}{4R_0},
\end{align*}
 for small $d>0$, 
  and define the $n$-dimensional cone in $\Omegacoord{\xbar_0}$ with angle $\rho_d$ by
 \begin{equation*}
 \cone_{\rho_d}:=\{p\in \Omegacoord{\xbar_0}\mid \lvert p\rvert\cos{\rho_d}\leq p^n\}.
 \end{equation*}
Also, 
define
 \begin{equation*}
\base_{d, \rho_d}:=\{p\in\Omegacoord{\xbar_0} \mid \lvert p''\rvert= d\} \cap\cone_{\rho_d}.
\end{equation*}
 Then from the same calculations leading up to~\eqref{eqn: rho d}  we obtain by~\eqref{eqn: flat contact set} that $$\base_{d, \rho_d} \cap \subzerocoord = \emptyset.$$
In particular,
 \begin{equation*}\label{eqn: strict positivity on boundary A3w case}
  \util(p)>0 \quad \text{for $p\in \base_{d, \rho_d}$}.
 \end{equation*}
Since $\Omega$ is bounded, as in the proof of Theorem~\ref{thm: euclidean thm},  we can define the strictly positive quantity
\begin{equation*}
f(d):=  \min\left[d^2, \min_{p\in \base_{d, \rho_d}}{\util(p)}\right]>0
\end{equation*}
for each $d>0$ sufficiently small. The definition of $f(d)$ differs from the proof of Theorem~\ref{thm: euclidean thm}, the reason being that we cannot rely on orthogonality to fully eliminate all of the extraneous terms as we did for the Euclidean cost.

 Before we construct the families of sets $\cyl_d$ and $\cylbar_d$ to be used in the volume comparison argument, we first give an appropriate parametrization of  the possibly curved set $\curvedvalley$. Recall that we are equating a point in $\Omegacoord{\xbar_0}$ with a point in $\R^n$ by its coordinates with respect to the basis $\{e_i\}$. Define the mapping
$F: \Omegacoord{\xbar_0}\to \R^k$ by
\begin{equation*}
F(p):=(\mountaintilde_{r_0\ebar_1}(p),\ldots, \mountaintilde_{r_0\ebar_k}(p)), 
\end{equation*}
and $\psi: \Omegacoord{\xbar_0}\to \R^n$ by
\begin{equation*}
\psi(p):= (F(p), p'').
\end{equation*}
 Note that by construction $F(0)=0$, thus $\psi (0)=0$. 
By the inverse function theorem, $\psi$ is a diffeomorphism for $p$ in some small neighborhood (depending only on $c$) of the origin (recall, which is contained in the interior of $\Omegacoord{\xbar_0}$),  and we can write the inverse as
\begin{equation*}
\psi^{-1}(p)=(G(p), p'')  
\end{equation*}
for some map $G: \R^n\to \{p\in\Omegacoord{\xbar_0}\mid p''=0\}$. Notice that  $ F(\psi^{-1} (p'') ) = 0 $ thus, 
 \begin{align}\label{eqn: in the curved valley}
\psi^{-1}(p'') \in \curvedvalley.
\end{align}
The vector $p - \psi^{-1}(p'') $ is in the span of the first $k$ basis vectors $e_i$, $1\le i\le k$, and gives the displacement of $p$ from the set $\curvedvalley$. This vector will be used crucially in Step 3. The size of this displacement is
\begin{align*}
\lvert p-\psi^{-1}(p'')\rvert&= \lvert p'-G(p'')\rvert \notag\\
&= \lvert G(\psi(p))-G(p'')\rvert \notag\\
&\sim_M \lvert \psi(p)-p''\rvert \notag\\
&= \lvert F(p)\rvert
\end{align*}
i.e.,
\begin{align}
M^{-1}\lvert F(p)\rvert\leq\lvert p-\psi^{-1}(p'')\rvert\leq M\lvert F(p)\rvert
\label{eqn: p-psi inverse bound}.
\end{align}
Also note that since $F$  and $G$ have Lipschitz constants depending only on $c$ and $\spt {\nu}$, we can calculate 
\begin{align}
\lvert p\rvert&\leq \lvert p'\rvert+\lvert p''\rvert\notag\\
&=  \lvert G(\psi(p))\rvert+ \lvert p''\rvert\notag\\
&\leq \lvert G(\psi(p))-G(p'')\rvert+\lvert G(p'')-G(0)\rvert+ \lvert p''\rvert\notag\\
&\leq M(\lvert \psi(p)-p''\rvert+\lvert p''\rvert)+ \lvert p''\rvert\notag\\
&\leq M(\lvert F(p)\rvert+ \lvert p''\rvert)\label{eqn: p bound}
\end{align} 
while since $\psi(0)=0$ we find
\begin{align}\label{eqn: psi inverse bound}
 \lvert \psi^{-1}(p)\rvert \sim_M \lvert p\rvert.
\end{align}

%
%
 Now, define the desired family of sets in $\Omega$ by
\begin{equation*}
\cyl_d:=\{p\in \Omegacoord{\xbar_0}\mid \lvert F(p)\rvert\leq f(d),\ \lvert p''\rvert \leq  d\}\cap\cone_{\rho_d}.
\end{equation*} 
If $d$ is sufficiently small, the set
$$\{p\in \Omegacoord{\xbar_0}\mid \lvert F(p)\rvert\leq  f(d),\ \lvert p''\rvert \leq   d\}$$ is diffeomorphic to the set 
\begin{equation*}
\{x\in \R^n\mid \lvert x'\rvert\leq M f(d),\ \lvert x''\rvert\leq  d\}, 
\end{equation*}
hence we have the volume bound
\begin{align}\label{eqn: cyld volume}
 \Leb{\cyl_d}\lesssim f(d)^kd^{n-k}.
\end{align}

 To define the other family of sets $\cylbar_d$ in $\Omegabar$, consider 
 \begin{equation*}
\conebar_{\constOne (\rho_d-\pi/2)}:=\{\pbar\in \Omegabarcoord{x_e}\mid \lvert \pbar''\rvert\cos\big({\constOne (\rho_d-\pi/2)}\big)\leq \pbar^n\},
\end{equation*}
 cylinder whose base is an $(n-k)$-dimensional cone in $\Omegabarcoord{x_e}$ with angle ${\constOne (\rho_d-\pi/2)}$.

Using this we define the set  (compare with \eqref{eqn: Euclidean Wbar d})
\begin{align}\label{eqn: general Wbar d}
\bar \base_d := \conebar_{\constOne (\rho_d-\pi/2)} \cap \Big\{ \pbar \in \Omegabarcoord{x_e} \mid \pbar'\in \conv{\{ 0,\ \constThree d\ebar_i\mid 1\leq i\leq k\}},\ \lvert\pbar''\rvert = \constFour \frac{f (d)}{d}\Big\} 
\end{align}
for appropriate small  positive constants $\constOne$, $\constThree$, and $\constFour$ to be determined, independent of $d$.
We now define the target shape $\cylbar_d$ as the cone over  $\bar\base_d$ with vertex $0$, that is,
\begin{equation*}
\cylbar_d:= \bigcup_{\pbar \in \bar\base_d} [0, \pbar] 
\end{equation*}
where $[0, \pbar]$ denotes the line segment (in the convex set $\Omegabarcoord{x_e}$) from $0$ to $\pbar$. By a calculation similar to~\ref{eqn: vol cylbar}, we see $\cylbar_d$ has the volume bound
\begin{align}\label{eqn: general cylbar vol}
 \left(\frac{f(d)}{d} \right)^{n-k} d^{n-k-1}d^k\sim \Leb{\cylbar_d}
\end{align}
%
%
%
%
%
%
 
 \subsubsection*{Step 3}
This step will be the most involved. 
We now claim that for some appropriate choice of $\constOne$, $\constThree$ and  $\constFour$ above (independent of $d$), 
for each $d>0$ sufficiently small we have 
\begin{equation}\label{eqn: show this inequality a3w case}
\mountaintilde_{\pbar}(p)\leq \util(p) \text{ for every $p\in \partial \cyl_d$ as long as  $\pbar\in \cylbar_d$}
\end{equation}
 (compare with~\eqref{eqn: show this inequality euclidean case}).
 Notice that  $\util \ge 0$ and $\mountaintilde_{0} \equiv 0$ (because $\ctil (\cdot, \xbar_0) \equiv 0$),  moreover, $\mountaintilde_{\pbar}$ satisfies  Loeper's maximum principle (Lemma~\ref{lem: DASM}) along each line segment in $\Omegabarcoord{x_e}$. Therefore,  by the definition of $\cylbar_d$ we see that it is sufficient to show~\eqref{eqn: show this inequality a3w case} for each $\pbar \in \bar\base_d$.

Fix $\pbar \in \bar\base_d$.
 As $$\partial \cyl_d \subset  \base_{d, \rho_d} \cup \{ p \in \Omegacoord{\xbar_0} \mid  |F(p)| = f(d)\} \cup  \partial \cone_{\rho_d},$$
there are three cases for $p\in \partial \cyl_d$:
\begin{enumerate}
\item $p\in \base_{d, \rho_d} \cap \partial \cyl_d$;
\item   $p \in \partial \cyl_d$ and $\lvert F(p)\rvert=f(d)$;
\item $p \in \partial\cone_{\rho_d}\cap \partial \cyl_d$.
\end{enumerate}

%
%
%
\ \\
\noindent {\em Case 1:} First suppose that $p\in \base_{d, \rho_d} \cap  \partial\cyl_d$, we then calculate for sufficiently small $d$, 
\begin{align*}
\mountaintilde_{\pbar}(p)&=[\mountaintilde_{\pbar}(p)-\mountaintilde_{\pbar'}(p)]+[\mountaintilde_{\pbar'}(p)-\mountaintilde_{\pbar'}(\psi^{-1}(p''))]+\mountaintilde_{\pbar'}(\psi^{-1}(p''))\\
&=I+II+\mountaintilde_{\pbar'}(\psi^{-1}(p'')).
\end{align*}
Then, using
~\eqref{eqn: p bound}, we find,
\begin{align*}
I&=-\ctil(p, \xbar_{\pbar})+\ctil(0, \xbar_{\pbar})-(-\ctil(p, \xbar_{\pbar'})+\ctil(0, \xbar_{\pbar'}))\\
&\leq M\lvert p\rvert \lvert\pbar-\pbar'\rvert\\
&\leq M (\lvert F(p)\rvert+ \lvert p''\rvert) \lvert \pbar''\rvert\\
&\leq   M(f(d)+ d)\left(\frac{\constFour f(d)}{d}\right) \quad \text{( since $\pbar \in \bar\base_d$, $|\pbar''| \le \frac{\constFour f(d)}{d}$ )}\\
&\leq   M\constFour f(d)
\end{align*}
for some $M>0$ depending only on bounds on derivatives of $c$.
Similarly, we see by using~\eqref{eqn: p-psi inverse bound} that
\begin{align*}
II&= -\ctil(p, \xbar_{\pbar'}) + \ctil(\psi^{-1}(p''), \xbar_{\pbar'})\\
&\leq M\lvert p-\psi^{-1}(p'')\rvert \lvert\pbar'\rvert  \quad \text{(since $\ctil (\cdot, \xbar_0) \equiv 0 $)} \\
&\leq  M\lvert F(p)\rvert  \lvert\pbar'\rvert \\
&\leq  M \constThree d f(d).
\end{align*} 
Finally, if $d$ is small enough, then for $\pbar \in \bar\base_d$,  $\pbar'$ is in the convex hull $ \conv\{r_0 \ebar_i \mid 1\le i \le k\}$.
\begin{align*}
\mountaintilde_{\pbar'}(\psi^{-1}(p''))
&\leq \max_{1\le i \le k}\{\mountaintilde_{r_0\ebar_i}(\psi^{-1}(p'')\}\\
&= 0
\end{align*}
by Lemma~\ref{lem: DASM} and (\ref{eqn: in the curved valley}). 
Hence all together for small constants $\constThree, \constFour$ (depending only on $M$)
 we obtain 
\begin{align*}
\mountaintilde_{\pbar}(p)&\leq M[\constThree d +\constFour ]f(d)\\
&\leq f(d)\\
&\leq \util(p)
\end{align*} 
where the last line follows from  
the definition of $f(d)$, since $p \in \base_{d, \rho_d}$.  The claim (\ref{eqn: show this inequality a3w case}) is verified in this first case.


\ \\
\noindent {\em Case 2:}  Next, suppose that $p\in \partial \cyl_d$, and $\lvert F(p)\rvert = f(d)$. 
 This case is subtle.   
In the following, the vector $p-\psi^{-1}(p'') $ will play an important role. Recall that it is contained in the span of the first $k$ basis vectors, $e_i$, $1\le i\le k$. 
Thus by the relations~\eqref{eqn: dual basis} and using~\eqref{A2}, for some choice of either plus or minus, and some index $1\leq j_0\leq k$, we find that  
\begin{align}\label{eqn: e j0 bound}
\pm \euclidean{(p-\psi^{-1}(p''))}{\ebar_{j_0}} \ge \frac{1}{M}\lvert p-\psi^{-1}(p'')\rvert,
\end{align}
where $M>0$ depends only on $c$ and $k$.
At the same time, we claim that 
\begin{align}\label{eqn: lower than plus minus}
\mountaintilde_{\pbar'}(p)\leq \mountaintilde_{\pm r_0\ebar_{j_0}}(p),\qquad \forall\; p\in\curvedvalley.
\end{align}
Indeed, note that $\pbar'\in\conv{\{0,\ r_0\ebar_i\mid 1\leq i\leq k\}}$ if $d$ is small enough. Thus by applying Loeper's maximum principle (Lemma~\ref{lem: DASM}) multiple times we see that 
\begin{align*}
 \mountaintilde_{\pbar'}(p)&\leq \max{\{\mountaintilde_0(p),\ \mountaintilde_{r_0\ebar_{j_0}}(p)\}}\\
 &=0\\
 &=\mountaintilde_{r_0\ebar_{j_0}}(p)
\end{align*}
if $p\in \curvedvalley$ (using that $\mountaintilde_0\equiv 0$). On the other hand, by property~\eqref{eqn: sublevelsets contained in halfspaces} we have
\begin{equation*}
\{p\in\Omegacoord{\xbar_0}\mid \mountaintilde_{- r_0\ebar_{j_0}}\leq 0\}\subset\{p\in\Omegacoord{\xbar_0}\mid -p^{j_0}\leq 0\}
\end{equation*}
hence by the continuity of $c$, 
\begin{equation*}
\curvedvalley\subset \{p\in \Omegacoord{\xbar_0}\mid p^{j_0}\leq 0\}\subset \{p\in \Omegacoord{\xbar_0}\mid 0\leq\mountaintilde_{ -r_0\ebar_{j_0}} \},
\end{equation*}
and we obtain the claimed inequality~\eqref{eqn: lower than plus minus}. It can be seen that we may assume the choice of vector is $+\ebar_{j_0}$ without any loss of further generality, so we shall do so for the remainder of this case.

Now our goal in this case is to show the inequality $\mountaintilde_{\pbar}(p) \le \mountaintilde_{r_0\ebar_{j_0}} (p)$. Indeed, $r_0\ebar_{j_0} \in \partial_{\ctil} \util(0)$ implies $\mountaintilde_{r_0\ebar_{j_0}} \le \util$, which then easily leads to the desired inequality~\eqref{eqn: show this inequality a3w case}. 
We assume in the following $d$ is sufficiently small. 
First write 
\begin{align*}
 \mountaintilde_{\pbar}(p)
 & = 
  \left[ \mountaintilde_{\pbar}(p)  -  \mountaintilde_{\pbar'}(p)\right] \\
  & \quad +   \left[\mountaintilde_{\pbar'}(p)  - \mountaintilde_{\pbar'}(\psi^{-1}(p''))\right] \\
  & \quad +  \left[-\mountaintilde_{r_0\ebar_{j_0}} (p) +\mountaintilde_{r_0 \ebar_{j_0}}(\psi^{-1}(p''))\right]\\
 & \quad + \left[\mountaintilde_{\pbar'}(\psi^{-1}(p''))  -\mountaintilde_{r_0 \ebar_{j_0}}(\psi^{-1}(p''))\right]\\
 & \quad +  \mountaintilde_{r_0 \ebar_{j_0}}(p)\\
 & = I + II  + III + IV+ \mountaintilde_{r_0 \ebar_{j_0}}(p) 
\end{align*}
Here, similarly to Case 1, we see
\begin{align*}
 I \le MC_1 f(d).
\end{align*}

To deal with the term $II$, we first make an auxiliary calculation, which will be of use in a number of places later. Fix any $p_1$, $p_2\in \Omegacoord{\xbar_0}$ and $\pbar\in \Omegabarcoord{x_e}$.  Then using a Taylor expansion argument along with~\eqref{A2},~\eqref{eqn: modified differential}, and~\eqref{eqn: dual basis} results in the bound 
\begin{align}
 &-\ctil(p_2, \xbar_{\pbar})  +\ctil(p_1, \xbar_{\pbar})\notag\\
 &\qquad\leq \langle-D\ctil(p_1, \xbar_{\pbar})),p_2-p_1\rangle+M\lvert \pbar\rvert\lvert p_2-p_1\rvert^2\notag\\
&\qquad\leq \langle-D\ctil(0, \xbar_{\pbar})),p_2-p_1\rangle+M\lvert \pbar\rvert\lvert p_1\rvert\lvert p_2-p_1\rvert+M\lvert \pbar\rvert\lvert p_2-p_1\rvert^2\notag\\
&\qquad=\euclidean{\pbar}{(p_2-p_1)}+M\lvert \pbar\rvert\lvert p_1\rvert\lvert p_2-p_1\rvert+M\lvert \pbar\rvert\lvert p_2-p_1\rvert^2.\label{eqn: 2nd order taylor expansion}
\end{align}
By applying~\eqref{eqn: 2nd order taylor expansion} to the term $II$ and using the estimates~\eqref{eqn: p bound} and~\eqref{eqn: psi inverse bound}, we obtain
\begin{align*}
 II&=-\ctil(p, \xbar_{\pbar'})+\ctil(0, \xbar_{\pbar'})-(-\ctil(\psi^{-1}(p''), \xbar_{\pbar'})+\ctil(0, \xbar_{\pbar'}))\\
 & \leq \pbar' \cdot (p - \psi^{-1}(p'')) +  M \lvert\psi^{-1}(p)\rvert\lvert p-\psi^{-1}(p'')\rvert+ M\lvert p-\psi^{-1}(p'')\rvert^2  \\
 & \leq \lvert\pbar' \rvert \lvert F(p)\rvert +  M \lvert p\rvert\lvert F(p)\rvert+ M\lvert F(p)\rvert^2  \\
& \le C_0 M d \, f(d) + M (f(d)+ d) f(d) + M f(d)^2 \\
& \qquad \text{(using $p\in \cyl_d$, $\pbar \in \bar \cyl_d$)}
 \end{align*}
To deal with the term $III$, first we note that since $\mountaintilde_{0} \equiv 0$, by~\eqref{A2} we will have
 \begin{align*}
\lvert D\mountaintilde_{\pbar}\rvert&\leq M \lvert\pbar\rvert.   
\end{align*}
Hence by~\eqref{eqn: 2nd order taylor expansion},~\eqref{eqn: p bound}, and~\eqref{eqn: psi inverse bound} again,
\begin{align*}
 III &=-\ctil(\psi^{-1}(p''), \xbar_{r_0\ebar_{j_0}})+\ctil(p, \xbar_{r_0\ebar_{j_0}})\\
 &\le -r_0\euclidean{\ebar_{j_0}}{(p -\psi^{-1}(p''))}  + M \lvert \psi^{-1}(p'')\rvert\lvert p -\psi^{-1}(p'')\rvert + M \lvert p -\psi^{-1}(p'')\rvert^2  \\
  & \le  -M^{-1}\lvert p-\psi^{-1}(p'')\rvert  + M \lvert p\rvert\lvert p -\psi^{-1}(p'')\rvert + M \lvert p -\psi^{-1}(p'')\rvert^2  \\
  & \qquad \text{(the first term by  \eqref{eqn: e j0 bound})}\\
 & \le  - M^{-1} f(d)  + M (f(d) + d) f(d) + M f(d)^2 \\
  & \qquad \text{(using $p\in \cyl_d$, $\pbar \in \bar \cyl_d$ and the condition $\lvert F(p)\rvert = f(d)$).} 
\end{align*}
The key here is that the first term $- M^{-1} f(d)$ in the last line above decays slower than all other terms in $I$, $II$ and $III$, if $\constFour$ is chosen to be small. 
In particular, for sufficiently small fixed $\constFour$ and for small $d$, we see
\begin{align*}
  I+ II + III \le 0.
\end{align*} 
Finally, since $\psi^{-1}(p'')\in \curvedvalley$ we have $IV \le 0$ by~\eqref{eqn: lower than plus minus}.
All together, this shows that $\mountaintilde_{\pbar} (p)\le \mountaintilde_{r_0\ebar_{j_0}} (p)$ as desired, and we obtain~\eqref{eqn: show this inequality a3w case} in this case.

\ \\
\noindent {\em Case 3:}

 Finally,  suppose that $p \in \partial \cone_{\rho_d} \cap \partial \cyl_d$. This case is subtle as well. First, notice that for small universal constants $\constOne$ and $\constThree$, and for small $d>0$, for $p \in  \cone_{\rho_d}$ and $\pbar \in \bar\base_d$, the angle between  $p$ and $\pbar''=(0,\pbar'')$ is greater than $\pi/2 + \gamma d$ with
\begin{align*}
\gamma:= \frac{(1-\beta)}{4R_0}.
\end{align*}
Note that by applying similar calculations as~\eqref{eqn: 2nd order taylor expansion} but in the variables $\pbar$, we obtain
\begin{align*}
 &[-\ctil(p, \xbar_{\pbar_2})+\ctil(0, \xbar_{\pbar_2})]-[-\ctil(p, \xbar_{\pbar_1})+\ctil(0, \xbar_{\pbar_1})]\\
 &\qquad\leq \euclidean{p}{(\pbar_2-\pbar_1)}+M\lvert p\rvert\lvert \pbar_1\rvert\lvert \pbar_2-\pbar_1\rvert+M\lvert p\rvert\lvert \pbar_2-\pbar_1\rvert^2
\end{align*}
Combining this with the lower bound on the angle above, we obtain
\begin{align}
&\mountaintilde_{\pbar}(p)-\mountaintilde_{\pbar'}(p)\notag\\
&=[-\ctil(p, \xbar_{\pbar})+\ctil(0, \xbar_{\pbar})]-[-\ctil(p, \xbar_{\pbar'})+\ctil(0, \xbar_{\pbar'})]\notag\\
&\leq p\cdot\pbar''+M\lvert p\rvert\lvert\pbar'\rvert\lvert\pbar''\rvert+M\lvert p\rvert\lvert \pbar''\rvert^2\notag\\
&\leq - \lvert p\rvert\lvert\pbar''\rvert \sin{(\gamma d)}+M\lvert p\rvert\lvert\pbar''\rvert\lvert\pbar\rvert \notag\\
&\quad \text{(from the lower bound on the angle between $p$ and $\pbar''$)}\notag\\
& \leq - M(1-\constOne)\lvert p\rvert\lvert\pbar''\rvert d+M\lvert p\rvert\lvert\pbar''\rvert\lvert\pbar\rvert\notag\\
&=\lvert p\rvert\lvert\pbar''\rvert\left(- M(1-\constOne) d+M\lvert\pbar\rvert \right)\notag\\
& \le  \left(\constOne+\constThree + \constFour-1\right)M\lvert p\rvert\lvert\pbar''\rvert d\notag\\
& \qquad \text{(since $\pbar \in \bar\base_d$ and $\lvert \pbar \rvert \le \constThree d + \constFour f(d)/d \le (\constThree + \constFour) d$ )}\notag\\
&\leq 0\label{eqn: outside of cone a3w case}
\end{align}
for sufficiently small  constants $\constOne, \constThree, \constFour$ (independent of $d$), and for sufficiently small $d$. 
As a result,
\begin{align*}
\mountaintilde_{\pbar}(p) \le \mountaintilde_{\pbar'}(p)\leq \util(p)
\end{align*}
where we have used the fact that $\pbar' \in \partial_{\ctil} \util (0)$ for $\lvert\pbar '\rvert \le r_0$. This shows \eqref{eqn: show this inequality a3w case} for the final case.  

With this we have shown the inequality~\eqref{eqn: show this inequality a3w case} for all $p\in \partial \cyl_d$ and $\pbar\in \bar \base_d$, since it clearly holds when $\pbar=0$ the inequality holds for all $\pbar \in  \cylbar_d$ by Lemma~\ref{lem: DASM}.

\subsubsection*{Step 4}
We now claim that $\cylbar_d\subset\partial_cu(\cyl_d)$ for $d$ sufficiently small. If we can first show that $\util-\mountaintilde_{\pbar}$ has a local minimum in the set $\cyl_d$ for each $\pbar\in \cylbar_d$, we may then apply Corollary~\ref{cor: local to global} to obtain the desired inclusion. Indeed, if there exists a point in the interior of $\cyl_d$ where $\util<\mountaintilde_{\pbar}$ we immediately obtain a local minimum from~\eqref{eqn: show this inequality a3w case}. If this does not hold, we see by the calculations leading to~\eqref{eqn: outside of cone a3w case} that $\mountaintilde_{\pbar}\leq \util$ on a small neighborhood of $0$. Since $\util(0)=0=\mountaintilde_{\pbar}(0)$ by construction, we see that $0$ is a local minimum of $\util-\mountaintilde_{\pbar}$.

%
%

\subsubsection*{Step 5}

Finally, by using~\eqref{eqn: upper bound} in conjunction with the volume bounds~\eqref{eqn: cyld volume} and~\eqref{eqn: general cylbar vol}, we see (for all $d>0$ sufficiently small) 
\begin{align*}
 d^{n-k-1}d^{k} \left(\frac{f(d)}{d}\right)^{n-k} &\sim \Leb{\cylbar_d}\\
&\lesssim\Leb{\partial_c u(\cyl_d)}\\
&\leq \Leb{\cyl_d}\\
&\lesssim f(d)^{k}d^{n-k}.
\end{align*}

However, since $n-2k\leq 0$ by assumption and $f(d)/d \le d$ by definition, we see that $(f(d)/d)^{n-2k}$ is bounded below away from zero, so after rearranging we obtain
\begin{equation*}
1\lesssim\left(\frac{f(d)}{d}\right)^{n-2k}\lesssim d.
\end{equation*}
 which is a contradiction as $d\to 0$. This completes the proof of Theorem~\ref{thm: A3w subdifferential bound}.
 \end{proof}

\appendix
\appendixpage
\section{An example of a purely \eqref{A3w}  cost}\label{section: a3w example}
In this section, we give an example of a cost function which satisfies~\eqref{A3w}, but not~\eqref{A3s} and~\eqref{NNCC}. This example originally appears in~\cite{TW09} where it is stated to satisfy~\eqref{A3w} but not~\eqref{A3s}. It was communicated to the authors by Neil Trudinger \cite{Trudinger} that the cost function actually fails to satisfy~\eqref{NNCC} as well. The proof given in this appendix is only for the reader's convenience, since the detailed proof does not seem to appear in the literature.  

\begin{lem}\label{lem: |x-y|^-2}
Define
\begin{align*}
 c(x, \xbar):=\lvert x-\xbar\rvert^{-2}
\end{align*}
for $x$, $\xbar\in\R^n$. 
Then, $c$ satisfies conditions~\eqref{A0}-\eqref{A2} and~\eqref{A3w}, but does not satisfy~\eqref{A3s} and~\eqref{NNCC} on any subdomain of $\R^n\times\R^n$ whose closure does not contain the diagonal $\{(x,\xbar)\mid x=\xbar\}$.
\end{lem}
\begin{proof}
It is clear that $c$ satisfies~\eqref{A0} on any such a subdomain of $\R^n\times \R^n$.

Assume throughout that $x\neq\xbar$. We can calculate, 
 
\begin{align*}
 -D c(x, \xbar)=2(x-\xbar)\vert x-\xbar\rvert^{-4},
\end{align*}
it is clear that for a fixed $x$, this is an injective mapping in $\xbar$, and by the symmetry of the two variables, we see that~\eqref{A1} is satisfied. Additionally,
\begin{align}
 D_i\Dbar_jc(x, \xbar)&=-D^2_{ij}c(x, \xbar)\notag\\
 &=2\lvert x-\xbar\rvert^{-4}[\delta_{ij}-4(x_i-\xbar_i)(x_j-\xbar_j)\lvert x-\xbar\rvert^{-2}]\label{eqn: c_ij}
\end{align}
hence
\begin{align*}
 \det{D\Dbar c(x, \xbar)}&=2^n\lvert x-\xbar\rvert^{-4n}\det{\left(\Id-\frac{4(x-\xbar)\otimes(x-\xbar)}{\lvert x-\xbar\rvert^{2}}\right)}\\
 &=-3\cdot 2^n\lvert x-\xbar\rvert^{-4n}\neq 0,
\end{align*}
hence we obtain~\eqref{A2}.

We will now verify condition~\eqref{A3w} by a lengthy, but routine calclation. Note that an alternate way of writing condition~\eqref{A3w} is by the inequality
\begin{align*}
 D^2_{p_kp_l}A_{ij}(x, p)V^iV^j\eta_k\eta_l\geq 0
\end{align*}
for $x\in\R^n$, $p\neq 0$, and $\langle V, \eta\rangle=0$, where
\begin{align*}
 A_{ij}(x, p):=-D^2_{ij}c(x, \cExp{x}{p}).
\end{align*}
Now for $p\in\R^n$, since by definition $-Dc(x, \cExp{x}{p})=p$, by taking the absolute value of both sides of this expression we find
\begin{align*}
2^{-\frac{1}{3}} \lvert p\rvert^\frac{1}{3}=\lvert x-\cExp{x}{p}\rvert ^{-1},
\end{align*}
which leads to the formula
\begin{align*}
x-\cExp{x}{p}=2^{\frac{1}{3}}\lvert p\rvert^{-\frac{4}{3}}p.
\end{align*}
Plugging this expression into~\eqref{eqn: c_ij}, we find
\begin{align*}
 A_{ij}(x, p)=2^{-\frac{1}{3}} \lvert p\rvert^\frac{4}{3}(\delta_{ij}-4p_ip_j\lvert p\rvert^{-2}).
\end{align*}
Differentiating once in $p$,
\begin{align*}
 &D_{p_k}A_{ij}(x, p)\\
 &=2^{-\frac{1}{3}} [\frac{4}{3}p_k\lvert p\rvert^{-\frac{2}{3}}(\delta_{ij}-4p_ip_j\lvert p\rvert^{-2})-4\lvert p\rvert^\frac{4}{3}(\delta_{ik}p_j\lvert p\rvert^{-2}+p_i\delta_{jk}\lvert p\rvert^{-2}-2p_ip_jp_k\lvert p\rvert^{-4})]\\
 &=2^{\frac{5}{3}}\lvert p\rvert^{-\frac{2}{3}}\left[\frac{1}{3}\delta_{ij}p_k-\delta_{ik}p_j-p_i\delta_{jk}+\frac{2}{3}p_ip_jp_k\lvert p\rvert^{-2}\right],
\end{align*}
and differentiating in $p$ again,
\begin{align*}
&D^2_{p_kp_l}A_{ij}(x, p)\\
&=2^{\frac{5}{3}}\left(-\frac{2}{3}\right)p_l\lvert p\rvert^{-\frac{8}{3}}\left[\frac{1}{3}\delta_{ij}p_k-\delta_{ik}p_j-p_i\delta_{jk}+\frac{2}{3}p_ip_jp_k\lvert p\rvert^{-2}\right]\\
&\qquad+2^{\frac{5}{3}}\lvert p\rvert^{-\frac{2}{3}}\left[\frac{1}{3}\delta_{ij}\delta_{kl}-\delta_{ik}\delta_{jl}-\delta_{il}\delta_{jk}\right.\\
&\qquad\left.+\frac{2}{3}\delta_{il}p_jp_k\lvert p\rvert^{-2}+\frac{2}{3}p_i\delta_{jl}p_k\lvert p\rvert^{-2}+\frac{2}{3}p_ip_j\delta_{kl}\lvert p\rvert^{-2}-\frac{4}{3}p_ip_jp_kp_l\lvert p\rvert^{-4}\right]\\
&=\frac{2^{\frac{5}{3}}}{3}\lvert p\rvert^{-\frac{2}{3}}\left[-\frac{2}{3}\delta_{ij}p_kp_l\lvert p\rvert^{-2}+2\delta_{ik}p_jp_l\lvert p\rvert^{-2}+2p_i\delta_{jk}p_l\lvert p\rvert^{-2}\right.\\
&\qquad\left.-\frac{4}{3}p_ip_jp_kp_l\lvert p\rvert^{-4}+\delta_{ij}\delta_{kl}-3\delta_{ik}\delta_{jl}-3\delta_{il}\delta_{jk}\right.\\
&\qquad\left.+2\delta_{il}p_jp_k\lvert p\rvert^{-2}+2p_i\delta_{jl}p_k\lvert p\rvert^{-2}+2p_ip_j\delta_{kl}\lvert p\rvert^{-2}-4p_ip_jp_kp_l\lvert p\rvert^{-4}\right]
\end{align*}
Hence, for any $V$, $\eta\in\R^n$ we find
\begin{align*}
 & D^2_{p_kp_l}A_{ij}(x, p)V^iV^j\eta_k\eta_l\\
 &=\frac{2^{\frac{5}{3}}}{3}\lvert p\rvert^{-\frac{2}{3}}\left[-\frac{2}{3}\lvert V\rvert^2\langle \frac{p}{\lvert p\rvert}, \eta\rangle^2+2\langle V, \eta\rangle \langle \frac{p}{\lvert p\rvert}, V\rangle\langle \frac{p}{\lvert p\rvert}, \eta\rangle+2\langle V, \eta\rangle \langle \frac{p}{\lvert p\rvert}, V\rangle\langle \frac{p}{\lvert p\rvert}, \eta\rangle\right.\\
 &\qquad\left.-\frac{4}{3}\langle \frac{p}{\lvert p\rvert}, V\rangle^2\langle \frac{p}{\lvert p\rvert}, \eta\rangle^2+\lvert V\rvert^2\lvert\eta\rvert^2-3\langle V, \eta\rangle^2-3\langle V, \eta\rangle^2\right.\\
&\qquad\left.+2\langle V, \eta\rangle\langle \frac{p}{\lvert p\rvert}, \eta\rangle\langle \frac{p}{\lvert p\rvert}, V\rangle+2\langle V, \eta\rangle\langle \frac{p}{\lvert p\rvert}, V\rangle\langle \frac{p}{\lvert p\rvert}, \eta\rangle+2\lvert \eta\rvert^2\langle \frac{p}{\lvert p\rvert}, V\rangle^2-4\langle \frac{p}{\lvert p\rvert}, V\rangle^2\langle \frac{p}{\lvert p\rvert}, \eta\rangle^2\right]\\
 &=\frac{2^{\frac{5}{3}}}{3}\lvert p\rvert^{-\frac{2}{3}}\left[\lvert V\rvert^2\lvert\eta\rvert^2+2\lvert \eta\rvert^2\langle \frac{p}{\lvert p\rvert}, V\rangle^2-\frac{2}{3}\lvert V\rvert^2\langle \frac{p}{\lvert p\rvert}, \eta\rangle^2-\frac{16}{3}\langle \frac{p}{\lvert p\rvert}, V\rangle^2\langle \frac{p}{\lvert p\rvert}, \eta\rangle^2\right.\\
 &\qquad\left.-6\langle V, \eta\rangle^2+8\langle V, \eta\rangle \langle \frac{p}{\lvert p\rvert}, V\rangle\langle \frac{p}{\lvert p\rvert}, \eta\rangle\right].
\end{align*}

We are now ready to verify~\eqref{A3w}. Let us write
\begin{align*}
 q:&=\frac{p}{\lvert p\rvert},\\
 \theta:&=\langle q, V\rangle,
\end{align*}
and assume that $\lvert V\rvert=\lvert \eta\rvert =1$, with $\langle V, \eta\rangle =0$. By solving the  constrained optimization problem of maximizing the function $q\mapsto \langle q, \eta\rangle^2+\langle q, V\rangle ^2$ subject to $\lvert q\rvert^2=1$ in $\R^n$, we see that 
\begin{align*}
 \lvert \langle q, \eta\rangle \rvert\leq \sqrt{1-\theta^2}
\end{align*}
and thus for $\langle V, \eta\rangle =0$,
\begin{align*}
 & D^2_{p_kp_l}A_{ij}(x, p)V^iV^j\eta_k\eta_l\\
 &=\frac{2^{\frac{5}{3}}}{3}\lvert p\rvert^{-\frac{2}{3}}\left(1+2\theta^2-\frac{2}{3}\langle q, \eta\rangle^2-\frac{16}{3} \theta^2\langle q, \eta\rangle^2\right)\\
&\geq \frac{2^{\frac{5}{3}}}{3}\lvert p\rvert^{-\frac{2}{3}}\left(1+2\theta^2-\frac{2}{3}(1-\theta^2)-\frac{16}{3} \theta^2(1-\theta^2)\right)\\
&= \frac{2^{\frac{5}{3}}}{9}\lvert p\rvert^{-\frac{2}{3}}\left(1-8\theta^2+16\theta^4\right).
\end{align*}
This expression attains the minimum value of zero for $\theta\in[0, 1]$ at $\theta=1/2$, hence we see that $c$ satisfies~\eqref{A3w}.

Next, by the calculation above, we see that if we choose $p\in\R^n\setminus \{0\}$ and $\langle V, \eta\rangle=0$ with $\lvert V\rvert=\lvert \eta\rvert=1$ so that $\theta=\pm\frac{1}{2}$ and $\langle q, \eta\rangle=\sqrt{1-\frac{1}{4}}=\frac{\sqrt{3}}{2}$ (for example, $V=e_1$, $\eta=e_2$ and $p=\frac{1}{2}e_1+\frac{\sqrt{3}}{2}e_2$), then we actually have 
\begin{align*}
 D^2_{p_kp_l}A_{ij}(x, p)V^iV^j\eta_k\eta_l=0,
\end{align*}
hence $c$ does not satisfy~\eqref{A3s}.

Finally, we show that $c$ does not satisfy~\eqref{NNCC}. Consider $V=\lambda\eta$ with $\lvert\eta\rvert=1$ and any $\lambda \neq0$. Then for any $p$ that is a nonzero multiple of $\eta$,
\begin{align*}
 & D^2_{p_kp_l}A_{ij}(x, p)V^iV^j\eta_k\eta_l\\
 &=\frac{2^{\frac{5}{3}}}{3}\lvert p\rvert^{-\frac{2}{3}}\lambda^2\left[1+2-\frac{2}{3}-\frac{16}{3}-6+8\right]\\
 &=-\frac{2^{\frac{5}{3}}}{3}\lvert p\rvert^{-\frac{2}{3}}\lambda^2\\
 &<0
\end{align*}
finishing the proof.
\end{proof}
\bibliography{mybiblio}
\bibliographystyle{plain}
\end{document}